\documentclass[a4paper,english]{article}%
\usepackage[T1]{fontenc}
\usepackage[latin9]{inputenc}
\usepackage{amsmath}
\usepackage{amssymb}
\usepackage{babel}
\usepackage{color}
\usepackage{amsfonts}
\usepackage{graphicx}%
\providecommand{\U}[1]{\protect\rule{.1in}{.1in}}
\makeatletter

\newcommand*{\E}
{\operatorname{E}}

\newtheorem{thm}{Theorem}
\newtheorem{lem}{Lemma}

\newtheorem{rem}{Remark}

\newtheorem{example}{Example}
\makeatother
\newenvironment{proof}[1][Proof]{\noindent\textbf{#1.} }{\ \rule{0.5em}{0.5em}}
\begin{document}

\title{Projected particle methods for solving McKean-Vlasov stochastic differential equations}
\author{Denis Belomestny$^{1} $ and John Schoenmakers $^{2}$}
\maketitle

\begin{abstract}
We propose a novel projection-based particle method for solving McKean-Vlasov
stochastic differential equations. Our approach is based on a projection-type
estimation of the marginal density of the solution in each time step. The
projection-based particle method leads in many situation to a significant
reduction of numerical complexity compared to the widely used kernel density
estimation algorithms. We derive strong convergence rates and rates of density
estimation. The convergence analysis, particularly in the case of linearly
growing coefficients, turns out to be rather challenging and requires some new
type of averaging technique. This case is exemplified by explicit solutions to
a class of McKean-Vlasov equations with affine drift. The performance of the
proposed algorithm is illustrated by several numerical examples.

\end{abstract}

\footnotetext[1]{Duisburg-Essen University, Thea-Leymann-Str. 9, D-45127
Essen, Germany, \texttt{denis.belomestny@uni-due.de}} \footnotetext[2]%
{Weierstrass Institute for Applied Analysis and Stochastics, Mohrenstr. 39,
10117 Berlin, Germany, \texttt{schoenma@wias-berlin.de}} \emph{Keywords:}
Mckean-Vlasov equations, particle systems, projection estimators, explicit
solutions.\newline\emph{2010 MSC:} 60H10, 60K35.

\section{Introduction}

Nonlinear Markov processes are stochastic processes whose transition functions
may depend not only on the current state of the process but also on the
current distribution of the process. These processes were introduced by
McKean~\cite{mckean1966class} to model plasma dynamics. Later nonlinear Markov
processes were studied by a number of authors; we mention here the books of
Kolokoltsov~\cite{kolokoltsov2010nonlinear} and
Sznitman~\cite{sznitman1991topics}. These processes arise naturally in the
study of the limit behavior of a large number of weakly interacting Markov
processes and have a wide range of applications, including financial
mathematics, population dynamics, and neuroscience (see, e.g.,
\cite{frank2004stochastic} and the references therein).

Let $[0,T]$ be a finite time interval and $(\Omega,\mathcal{F},\mathrm{P})$ be
a complete probability space, where a standard $m$-dimensional Brownian motion
$W$ is defined. We consider a class of McKean-Vlasov SDEs, i.e. stochastic
differential equation (SDE) whose drift and diffusion coefficients may depend
on the current distribution of the process of the form:
\begin{equation}
\left\{
\begin{array}
[c]{ll}%
X_{t} & =\xi+\int_{0}^{t}\int_{\mathbb{R}^{d}}a(X_{s},y)\mu_{s}(dy)ds+\int
_{0}^{t}\int_{\mathbb{R}^{d}}b(X_{s},y)\mu_{s}(dy)dW_{s}\\
\mu_{t} & =\mathrm{Law}(X_{t}),\quad t\in\lbrack0,T],
\end{array}
\right.  \label{eq:sde}%
\end{equation}
where $X_{0}=\xi$ is an $\mathcal{F}_{0}$-measurable random variable in
$\mathbb{R}^{d},$ $a:\,\mathbb{R}^{d}\times\mathbb{R}^{d}\rightarrow
\mathbb{R}^{d}$ and $b:\,\mathbb{R}^{d}\times\mathbb{R}^{d}\rightarrow
\mathbb{R}^{d\times m}.$ {If the functions $a$ and $b$ are smooth with
uniformly bounded derivatives and the random variable $\xi$ has finite moments
of any order, then (see \cite{antonelli2002rate})} there is a unique strong
solution of \eqref{eq:sde} such that for all $p>1,$
\begin{equation}
\mathrm{E}\left[  \sup_{s\leq T}|X_{s}|^{p}\right]  \leq\infty.
\label{eq: xmom}%
\end{equation}
In the sequel we assume that there exists a unique strong solution of
\eqref{eq:sde} such that \eqref{eq: xmom} holds and refer to
\cite{funaki1984certain} for more general sufficient conditions for this.

Assume that $d=1$ and for any $t\geq0,$ the measure $\mu_{t}(du)$ possesses a
bounded density $\mu_{t}(u).$ Then the family of these densities satisfies a
nonlinear Fokker-Planck equation of the form
\begin{align}
\frac{\partial\mu_{t}(x)}{\partial t}  &  =-\frac{\partial}{\partial x}\left(
\left(  \int a(x,y)\mu_{t}(y)\,dy\right)  \,\mu_{t}(x)\right) \nonumber\\
&  +\frac{1}{2}\frac{\partial^{2}}{\partial x^{2}}\left(  \left(  \int
b(x,y)\mu_{t}(y)\,dy\right)  ^{2}\,\mu_{t}(x)\right)  , \label{eq:FP}%
\end{align}
which can be seen as an analogue of a well-known linear Fokker-Planck equation
in the case of linear stochastic differential equations. In
Section~\ref{expls} we will show that if the drift $a$ is affine in $x,$ and
the diffusion coefficient $b$ is independent of $x,$ then the system
(\ref{eq:sde}), and hence (\ref{eq:FP}), has an explicit solution. These
solutions, apart from being interesting in their own right, also provide
explicit cases of an explosive behavior.

The theory of the propagation of chaos developed in \cite{sznitman1991topics},
states that \eqref{eq:sde} is a limiting equation of the system of stochastic
interacting particles (samples) with the following dynamics
\begin{equation}
X_{t}^{i,N}=\xi^{i}+\int_{0}^{t}\int_{\mathbb{R}^{d}}a(X_{s}^{i,N},y)\mu
_{s}^{N}(dy)\,ds+\int_{0}^{t}\int_{\mathbb{R}^{d}}b(X_{s}^{i,N},y)\mu_{s}%
^{N}(dy)\,dW_{s}^{i} \label{eq:par}%
\end{equation}
for $i=1,\ldots,N,$ where $\mu_{t}^{N}=\frac{1}{N}\sum_{i=1}^{N}\delta
_{X_{t}^{i,N}},$ $\xi^{i},$ $i=1,\ldots,N,$ are i.i.d copies of $\xi,$
distributed according the law $\mu_{0},$ and $W^{i},$ $i=1,...,N,$ are
independent copies of $W.$ In fact it can be shown, under sufficient
regularity conditions on the coefficients, that convergence in law for
empirical measures on the path space holds, i.e., $\mu^{N}=\{\mu_{t}^{N}%
:t\in\lbrack0,T]\}\rightarrow\mu$, $N\rightarrow\infty$, see
\cite{meleard1996asymptotic}.

Despite the numerous branches of research on stochastic particle systems,
results on numerical approximations of McKean-Vlasov-SDEs are very sparse. The
authors in \cite{antonelli2002rate} proposed to use the Euler scheme with
time-step $h=T/L,$ that for $l=0,\ldots,L-1$, yields
\begin{equation}
\bar{X}_{t_{l+1}}^{i,N}=\bar{X}_{t_{l}}^{i,N}+\frac{1}{N}\sum_{j=1}^{N}%
a(\bar{X}_{t_{l}}^{i,N},\bar{X}_{t_{l}}^{j,N})\,h+\frac{1}{N}\sum_{j=1}%
^{N}b(\bar{X}_{t_{l}}^{i,N},\bar{X}_{t_{l}}^{j,N})\,\Delta_{l+1}W^{i}
\label{eq:sEuler}%
\end{equation}
for $i=1,\ldots,N,$ $t_{l}=hl,$ and $\Delta_{l+1}W^{i}=W_{h(l+1)}^{i}%
-W_{hl}^{i}$ {(see also \cite{bossy1997stochastic} for more general MVSDEs
and \cite{kloeden2017gauss} for Gauss-quadrature based approach)}.
Implementation of the above algorithm requires usually $N^{2}\times L$
operations in every step of the Euler scheme. By using the algorithm presented
here one can significantly reduce the complexity of the particle simulation
especially if the coefficients of the corresponding McKean-Vlasov SDE are
smooth enough.

The contribution of this paper is twofold. On the one hand, we propose a new
approximation methodology based on a projection-type estimation of the
marginal densities of \eqref{eq:sde}. This methodology often leads to
numerically more efficient algorithms than the kernel-type approximation
algorithms, as they can profit from a global smoothness of coefficients
$a,$$b$ and the corresponding marginal densities. On the other hand, we
present a comprehensive convergence analysis of the proposed algorithms in the
case of possibly linearly growing (in $x$) coefficients $a$ and $b.$ To the
best of our knowledge, no stability analysis of MVSDEs under this linear
growth assumption was done before. In fact such analysis is rather challenging
and requires a special type of averaging technique. And, last but not least,
we study a general class of MVSDEs with affine drift and derive their explicit
solutions, to the best of our knowledge, for the first time.

The paper is organized as follows. In Section~\ref{seq:ppm} we present the
idea of our projected particle method. Section~\ref{seq:conv} is devoted to
the convergence analysis of the projected particle method. In particular, in
Section~\ref{sec:dens} we derive the convergence rates for the corresponding
projected density estimate. Section~\ref{sec:affine} presents a thorough study
of affine MVSDEs. Numerical examples for affine and convolution-type MVSDEs
are presented in Section~\ref{sec:num}. All proofs are collected in
Section~\ref{sec:proofs}.

\section{Projected particle method}

\label{seq:ppm} Let $w:\mathbb{R}^{d}\rightarrow\mathbb{R}_{+}$ be some weight
function with $w>0,$ such that%
\[
a(x,\cdot),b(x,\cdot)\in L_{2}(\mathbb{R}^{d},w)\text{ \ \ for any }%
x\in\mathbb{R}^{d}.
\]
Let further $(\varphi_{k},$ $k=0,1,2,..)$ be a total orthonormal system in
\thinspace$L_{2}(\mathbb{R}^{d},w).$ The corresponding (generalized) Fourier
coefficients of the functions $a(x,\cdot)$ and $b(x,\cdot)$ are given by
\begin{align}
\alpha_{k}(x)  &  :=\int a(x,u)\varphi_{k}(u)\,w(u)\,du\in\mathbb{R}%
^{d},\label{abd}\\
\beta_{k}(x)  &  :=\int b(x,u)\varphi_{k}(u)\,w(u)\,du\in\mathbb{R}^{d\times
m}\nonumber
\end{align}
and the following series representation holds
\[
a(x,\cdot)=\sum_{k=0}^{\infty}\alpha_{k}(x)\varphi_{k}(\cdot)\text{ \ and
\ \ }b(x,\cdot)=\sum_{k=0}^{\infty}\beta_{k}(x)\varphi_{k}(\cdot),\text{
\ \ }x\in\mathbb{R}^{d},
\]
in $L_{2}(\mathbb{R}^{d},w).$ Further it is assumed that each function
$\varphi_{k}$ is bounded so that the functions
\begin{equation}
\gamma_{k}(s):=\mathrm{E}\left[  \varphi_{k}(X_{s})\right]  \label{gam}%
\end{equation}
are well defined. Now let us fix some natural number $K>0$ and consider a
\textit{projected particle approximation} for \eqref{eq:sde}
\begin{equation}
X_{t}^{i,K,N}=\xi^{i}+\int_{0}^{t}\sum_{k=0}^{K}\gamma_{k}^{N}(s)\alpha
_{k}(X_{s}^{i,K,N})\,ds+\int_{0}^{t}\sum_{k=0}^{K}\gamma_{k}^{N}(s)\beta
_{k}(X_{s}^{i,K,N})\,dW_{s}^{i} \label{eq:par_proj}%
\end{equation}
for $i=1,\ldots,N,$ where
\begin{equation}
\gamma_{k}^{N}(s):=\frac{1}{N}\sum_{j=1}^{N}\varphi_{k}(X_{s}^{j,K,N})
\label{heu}%
\end{equation}
can be regarded as an approximation to (\ref{gam}). The projected system
(\ref{eq:par_proj}), with (\ref{heu}), is heuristically motivated by assuming
that for any $s\geq0$ the measure $\mu_{s}(du)$ possesses a density $\mu
_{s}(u)$ that, in view of%
\[
\gamma_{k}(s)=\int\mu_{s}(u)\varphi_{k}(u)du
\]
(cf. (\ref{gam})), formally satisfies%
\begin{equation}
\mu_{s}(u)=\sum_{k=0}^{\infty}\gamma_{k}(s)\varphi_{k}(u)w(u). \label{formu}%
\end{equation}
Then, we (formally) have the expansion
\[
\int a(x,u)\mu_{s}(u)du=\sum_{k=0}^{\infty}\alpha_{k}(x)\gamma_{k}(s),
\]
and this motivates the drift term in (\ref{eq:par_proj}). For the diffusion
term in (\ref{eq:par_proj}) an analogue motivation applies. In order to solve
(\ref{eq:par_proj}) we may consider, for any fixed $L>0,$ an Euler-type
approximation,
\begin{gather}
\bar{X}_{t}^{i,K,N}=\bar{X}_{\eta(t)}^{i,K,N}+\sum_{k=0}^{K}\gamma_{k}%
^{N}\left(  \eta(t)\right)  \,\alpha_{k}\bigl(\bar{X}_{\eta(t)}^{i,K,N}%
\bigr)\,(t-\eta(t))\label{eq: EulerPartProj}\\
+\sum_{k=0}^{K}\gamma_{k}^{N}\left(  \eta(t)\right)  \,\beta_{k}\bigl(\bar
{X}_{\eta(t)}^{i,K,N}\bigr)\,(W_{t}^{i}-W_{\eta(t)})\nonumber
\end{gather}
for $i=1,\ldots,N,$ and $h=T/L,$ where $\eta(t):=lh$ for $t\in\lbrack
lh,(l+1)h),$ $l=1,\ldots,L.$ Note that in order to generate a discretized
particle system $(\bar{X}_{hl}^{i,K,N}),$ $i=1,\ldots,N,$ $l=1,\ldots,L,$ we
need to perform (up to a constant depending on the dimension) $NLK$
operations. This should be compared to $N^{2}L$ operations in
\eqref{eq:sEuler}. Thus if $K$ is much smaller than $N,$ we get a significant
cost reduction. {Of course, this complexity analysis implicitly assumes that
the generalized Fourier coefficients $\alpha_{k}(x)$ and $\beta_{k}(x)$ are
known in closed form or can be cheaply computed. For more details in this
respect see Remark~\ref{disc} below. }

\begin{rem}
\label{disc} Many well known McKean-Vlasov type models used in physics and
engineering are constructed and formulated {via certain Fourier type
expansions of the respective drift and/or diffusion coefficients.} For
example, in the famous Kuramoto-Shinomoto-Sakaguchi model (see e.g.
\cite{frank2004stochastic}, eq. (5.214)) or in the coupled Brownian phase
oscillators (see \cite{kostur2002nonequilibrium}) the mean field potential is
given by its Fourier series, which entails a similar expansion for the
coefficient $a(x,u)$ ($b(x,u)$ is constant).  Let us also mention
a classical work of \cite{drozdov1996expansion}, where a known power series
expansion for the coefficients of a nonlinear Fokker-Planck equation is
assumed. From another point of view, since the basis $\left( \varphi
_{k}\right)  $ with the corresponding weight $w$ can, in principle, be chosen
freely, it is natural to assume  they can be chosen such that the coefficients $\alpha_{k}(x)$ and
$\beta_{k}(x)$ can be computed in closed form. In this respect, let us give
some further examples. If for any $x,$ $a(x,\cdot)$ is a linear combination of
functions of the form:%
\[
q_{1}(u_{1})\cdot\cdot\cdot q_{d}(u_{d})
\]
where each $q_{i}:\mathbb{R\rightarrow R}$ is a polynomial with coefficients
possibly depending on $x$, then $\left(  \varphi_{k}\right)  $ may be taken to
be Hermite functions in $\mathbb{R}^{d},$ i.e.%
\[
\varphi_{\alpha}\left(  u\right)  =H_{\alpha_{1}}(u_{1})\cdot\cdot\cdot
H_{\alpha_{d}}(u_{d})e^{-\left\vert u\right\vert ^{2}/2},\quad\alpha
=(\alpha_{1},\ldots,\alpha_{d}).
\]
The latter situation appears for instance in the popular interaction case with
$a(x,u)=A(x-u),$ where the function $A$ has a given representation
\[
A(z)=\sum_{\alpha}c_{\alpha}z_{1}^{\alpha_{1}}\cdot\ldots\cdot z_{d}%
^{\alpha_{d}},\quad z\in\mathbb{R}^{d},\text{ \ \ }\alpha\in\mathbb{N}_{0}%
^{d}.
\]
As another example, note that the Fourier coefficients of any function of the
form
\[
u\rightarrow u_{1}^{\alpha_{1}}\cdot\ldots\cdot u_{d}^{\alpha_{d}}%
e^{-\frac{\left\vert u-c\right\vert ^{2}}{\sigma}},\quad c\in\mathbb{R}%
^{d},\text{ \ \ }\alpha\in\mathbb{N}_{0}^{d},\text{ }\sigma>0,
\]
with respect to the Hermite basis above can be expressed in closed form. One
so could also consider $a(x,u),$ $b(x,u)$ of the form%
\[
\sum_{r=1}^{R}q_{r}(x)u^{\alpha_{r}}e^{-\left\vert u-c_{r}(x)\right\vert
^{2}/\sigma_{r}(x)},
\]
with free to choose $q_{r}(x)\in\mathbb{R},$ $c_{r}(x)\in\mathbb{R}^{d},$
$\sigma_{r}(x)\in\mathbb{R}_{+},$ $\alpha_{r}\in\mathbb{N}_{0}^{d},$
$R\in\mathbb{N}.$
\end{rem}

\section{Convergence analysis}

\label{seq:conv} In this section we first study the convergence of the
approximated particle system \eqref{eq:par_proj} to the solution of the
original system \eqref{eq:sde}. As a first obvious but important observation,
we note that the distribution of the triple $\left(  X_{s}^{j,K,N},X_{s}%
^{K,N},X_{s}^{j}\right)  $ with $X_{s}^{K,N}:=\left(  {X}_{s}^{1,K,N}%
,\ldots,{X}_{s}^{N,K,N}\right)  $ does not depend on $j,$ and therefore we can
write%
\begin{equation}
\left(  X^{j,K,N},X^{K,N},X^{j}\right)  \overset{\text{distr.}}{=}\left(
X^{\cdot,K,N},X^{K,N},X^{\cdot}\right)  \text{ \ for }j=1,...,N. \label{id}%
\end{equation}
For ease of notation, henceforth we denote with $\left\vert \cdot\right\vert
:=\left\vert \cdot\right\vert _{\dim}$ for a generic dimension $\dim$ the
standard Euclidian norm in $\mathbb{R}^{\dim}.$ Let us make the following assumptions.

\begin{description}
\item[(AF)] The basis functions $(\varphi_{k})$ fulfil
\[
\left\vert \varphi_{k}(z)-\varphi_{k}(z^{\prime})\right\vert \leq
L_{k,\varphi}\left\vert z-z^{\prime}\right\vert ,\quad\left\vert \varphi
_{k}(z)\right\vert \leq D_{k,\varphi},\quad k=0,1,\ldots
\]
for all $z,z^{\prime}\in\mathbb{R}^{d}$ and some constants $L_{k,\varphi
},D_{k,\varphi}>0.$

\item[(AC)] The functions $\alpha_{k}(x),$ $\beta_{k}(x),$ $k=0,1,2,\ldots$
satisfy
\begin{align*}
\left\vert \alpha_{k}(x)\right\vert  &  \leq A_{k,\alpha}(1+\left\vert
x\right\vert )\text{ \ \ with \ \ }\left(  A_{k,\alpha}\right)  _{k=0,1,\ldots
}\in l_{2},\\
\sum_{k=0}^{\infty}D_{k,\varphi}A_{k,\alpha}  &  \leq D_{\varphi}A_{\alpha
},\text{ \ \ and \ \ }\sum_{k=0}^{\infty}L_{k,\varphi}A_{k,\alpha}\leq
L_{\varphi}A_{\alpha},\\
\left\vert \beta_{k}(x)\right\vert  &  \leq A_{k,\beta}(1+\left\vert
x\right\vert )\text{ \ \ with \ \ }\left(  A_{k,\beta}\right)  _{k=0,1,\ldots
}\in l_{2},\\
\text{ \ }\sum_{k=0}^{\infty}D_{k,\varphi}A_{k,\beta}  &  \leq D_{\varphi
}A_{\beta},\text{ \ \ and \ \ }\sum_{k=0}^{\infty}L_{k,\varphi}A_{k,\beta}\leq
L_{\varphi}A_{\beta},
\end{align*}
for some constants $A_{\alpha},$ $A_{\beta},$ $D_{\varphi},$ and $L_{\varphi
}>0,$ and further%
\begin{align*}
\sup_{x,x^{\prime}\in\mathbb{R}^{d},\text{ }x\neq x^{\prime}}\frac{|\alpha
_{k}(x)-\alpha_{k}(x^{\prime})|}{|x-x^{\prime}|}  &  \leq B_{k,\alpha}\text{
\ \ with \ \ }\sum_{k=0}^{\infty}D_{k,\varphi}B_{k,\alpha}\leq D_{\varphi
}B_{\alpha},\\
\sup_{x,x^{\prime}\in\mathbb{R}^{d},\text{ }x\neq x^{\prime}}\frac{|\beta
_{k}(x)-\beta_{k}(x^{\prime})|}{|x-x^{\prime}|}  &  \leq B_{k,\beta} \text{
\ \  with \ \ }\sum_{k=0}^{\infty}D_{k,\varphi}B_{k,\beta}\leq D_{\varphi
}B_{\beta},
\end{align*}
for some $B_{\alpha},$ $B_{\beta}$ $>0.$
\item[(\(\mbox{AM}_p\))] For some \(p>0\) the initial distribution \(\mu_0\) possesses a finite absolute moment of order \(p.\)

\end{description}

In the sequel, for any random variable $\xi\in\mathbb{R}^{\dim}$ on
$(\Omega,\mathcal{F},\mathrm{P})$ we shall use $\Vert\xi\Vert_{p}$ for the
norm of $\left\vert \xi\right\vert $ in $L_{p}(\Omega).$ The following bound
on the strong error can be proved.

\begin{thm}
\label{thmc} For $p\geq2,$ it holds under assumptions (AC), (AF) and (\(AM_{p}\))
that
\begin{align}
\left\Vert \sup_{0\leq r\leq T}\left\vert X_{r}^{\cdot,K,N}-X_{r}^{\cdot
}\right\vert \right\Vert _{p}  &  \lesssim N^{-1/2}+\sum_{k=K+1}^{\infty
}A_{k,\alpha}\,\Vert\gamma_{k}\Vert_{L_{p}[0,T]}\nonumber\\
&  +\sum_{k=K+1}^{\infty}A_{k,\beta}\,\Vert\gamma_{k}\Vert_{L_{p}[0,T]},
\label{eq: main_bound}%
\end{align}
where $\lesssim$ stands for an inequality with some {(hidden)} positive finite
constant depending {only} on $A_{\alpha},A_{\beta},B_{\alpha},B_{\beta
,}D_{\varphi},L_{\varphi},$ $p,$ and $T.$
\end{thm}

\begin{rem}
For $1\leq p^{\prime}\leq2,$ we simply have%
\begin{equation}
\left\Vert \sup_{0\leq r\leq T}\left\vert X_{r}^{\cdot,K,N}-X_{r}^{\cdot
}\right\vert \right\Vert _{p^{\prime}}\leq\left\Vert \sup_{0\leq r\leq
T}\left\vert X_{r}^{\cdot,K,N}-X_{r}^{\cdot}\right\vert \right\Vert _{p}
\label{pk2}%
\end{equation}
for any $p\geq2.$
\end{rem}

The next theorem, on the convergence of the Euler approximation
(\ref{eq: EulerPartProj}) to the projected system \eqref{eq:par_proj}, can be
proved along the same lines as the proof of Theorem~\ref{thmc}.

\begin{thm}
For $p\geq2,$ it holds under assumptions (AC), (AF) and (\(AM_{p}\)) that for any
natural $K,N$
\[
\left\Vert \sup_{0\leq r\leq T}\left\vert \bar{X}_{r}^{\cdot,K,N}-X_{r}%
^{\cdot,K,N}\right\vert \right\Vert _{p}\lesssim\sqrt{h},
\]
where $\lesssim$ stands for an inequality with some {(hidden)} positive finite
constant depending only on $A_{\alpha},A_{\beta},B_{\alpha},B_{\beta
},D_{\varphi},L_{\varphi},$ $p$ and $T.$
\end{thm}

\paragraph{Discussion}

The bound \eqref{eq: main_bound} is proved under rather general assumptions on
the coefficients $a(x,y)$ and $b(x,y).$ In particular, we allow for linear
growth of these coefficients in $x$. This makes the proof of the bound in
Theorem~\ref{thmc} rather challenging, since we need to avoid an explosion. In
order to overcome this problem, we employ a kind of averaging technique which,
being combined with the symmetry of the particle distribution and the
existence of moments (see Section~\ref{sec:moments}), gives the desired bound.
Note that for this we have to assume existence and uniqueness of a strong
solution of the original MVSDE \eqref{eq:sde}. Funaki \cite{funaki1984certain}
proved existence and uniqueness under (essentially) global Lipschitz
condition. However, one should be able to extend his results by exploiting a
kind of one sided Lipschitz condition like in \cite{higham2002strong} or
\cite{hutzenthaler2012strong}.

The bound \eqref{eq: main_bound} consists of stochastic and approximation
errors. While the first error is of order $1/\sqrt{N},$ the second one depends
on $K$ and the properties of the coefficients $a(x,y)$ and $b(x,y).$ If these
coefficients are smooth in the sense that their generalized Fourier
coefficients $(\alpha_{k})$ and $(\beta_{k})$ decay fast, then the
approximation error can be made small even for medium values of $K.$

\begin{example}
The (normalized) Hermite polynomial of order $j$ is given, for $j\geq0$, by
\[
\overline{H}_{j}(x)=c_{j}(-1)^{j}e^{x^{2}}\frac{d^{j}}{dx^{j}}(e^{-x^{2}%
}),\quad c_{j}=\left(  2^{j}j!\sqrt{\pi}\right)  ^{-1/2}.
\]
These polynomials satisfy: $\int_{{\mathbb{R}}}\overline{H}_{j}(x)\overline
{H}_{\ell}(x)e^{-x^{2}}dx=\delta_{j,\ell}$ and, as a consequence,
\begin{equation}
\varphi_{k}(u)=\overline{H}_{k}(u)e^{-u^{2}/2},\quad k=1,2,\ldots,
\label{eq:herm_func}%
\end{equation}
is a total orthonormal system in $L_{2}\left(  \mathbb{R}^{d}\right)  $ (i.e.
here $w=1$). Moreover, $(\varphi_{k})_{k\geq0}$ fulfil the assumption (AF)
with $D_{k,\varphi}$ and $L_{k,\varphi}$ being uniformly bounded in $k,$ {see,
e.g. \cite{szego1975orthogonal}, p. 242.} Now let us suppose that
$a(x,\cdot),b(x,\cdot)\in L_{2}(\mathbb{R}^{d})$ for any $x\in\mathbb{R},$ and
discuss the assumptions (AC).
\end{example}

\begin{lem}
Suppose that for any $x\in\mathbb{R},$ the functions (in $u$)
\[
\widetilde a(x,u):=\frac{a(x,u)}{\sqrt{1+x^{2}}}, \quad\widetilde
b(x,u):=\frac{b(x,u)}{\sqrt{1+x^{2}}}
\]
admit derivatives in $u$ up to order $s>2$ such that the functions (in $u$)
\[
u^{\ell}\partial_{u}^{m}\widetilde a(x,u),\quad u^{\ell}\partial_{u}%
^{m}\widetilde b(x,u), \quad0\leq l+m\leq s
\]
are bounded and belong to $L_{1}({\mathbb{R}})$ (uniformly in $x$) together
with their first derivatives in $x.$ Then the assumption (AC) is satisfied
and
\begin{equation}
\left\Vert \sup_{0\leq r\leq T}\left\vert X_{r}^{\cdot,K,N}-X_{r}^{\cdot
}\right\vert \right\Vert _{p}\lesssim K^{1-s/2}+N^{-1/2}, \label{err}%
\end{equation}
as $K,N\rightarrow\infty.$
\end{lem}

\begin{proof}
We have
\begin{align*}
\alpha_{k}(x)  &  =\sqrt{1+x^{2}}\int\widetilde{a}(x,u)\overline{H}%
_{k}(u)\,e^{-u^{2}/2}\,du =\sqrt{1+x^{2}}\,\widetilde{\alpha}_{k}(x) ,\\
\beta_{k}(x)  &  =\sqrt{1+x^{2}}\int\widetilde{b}(x,u)\overline{H}%
_{k}(u)\,e^{-u^{2}/2}\,du= \sqrt{1+x^{2}}\,\widetilde{\beta}_{k}(x).
\end{align*}
The identity
\[
(2k+2)^{1/2}\overline{H}_{k}(x)=\overline{H}^{\prime}_{k+1}(z)
\]
and the integration-by-parts formula imply
\begin{align*}
\widetilde{\alpha}_{k}(x)  &  =\left.  \frac{\widetilde{a}(x,u)e^{-u^{2}%
/2}\overline{H}_{k+1}(u)}{(2k+2)^{1/2}}\right|  _{-\infty}^{\infty}\\
&  -\frac{1}{(2k+2)^{1/2}}\int_{-\infty}^{\infty}\left[  \frac{\partial
\widetilde{a}(x,u)}{\partial u}-u\widetilde{a}(x,u)\right]  \overline{H}%
_{k+1}(u)e^{-u^{2}/2}\,du.
\end{align*}
Note that $|\overline{H}_{k}(u)|\,e^{-u^{2}/2}\leq1$ uniformly in $u$ and $k$
({see, e.g. \cite{szego1975orthogonal}, p. 242.}) Hence if $\widetilde
{a}(x,u)$ is bounded and
\[
\int\left|  \frac{\partial\widetilde{a}(x,u)}{\partial u}-u\widetilde
{a}(x,u)\right|  \,du
\]
is bounded uniformly in $x,$ then $\widetilde{\alpha}_{k}(x)=O\left(
k^{-1/2}\right)  $ uniformly in $x.$ The second integration-by-parts shows
that $\widetilde{\alpha}_{k}(x)=O\left(  k^{-1}\right)  ,$ provided the
functions
\[
u^{2}\cdot\widetilde{a}(x,u),\frac{\partial^{2}\widetilde{a}(x,u)}{\partial
u^{2}},u\cdot\frac{\partial\widetilde{a}(x,u)}{\partial u}
\]
are integrable on $\mathbb{R}$ with their $L_{1}(\mathbb{R})$ norms uniformly
bounded in $x.$ Integrating by parts further, we derive the desired statement.
\end{proof}

\begin{rem}
As a rule, one chooses $N$ and $K$ such that the errors in (\ref{err}) are
balanced, that is $N^{1/(s-2)}\sim K,$ yielding a proportional reduction of
computational cost of order $N\cdot K/N^{2}$ $\sim N^{-(s-3)/(s-2)}.$
Alternatively we can compare the complexity, that is the computational cost
for achieving a prescribed accuracy $\varepsilon,$ for of the Euler schemes
(\ref{eq:sEuler}) and (\ref{eq: EulerPartProj}). It is not difficult to see that, after incorporating the
path-wise time discretization error, the standard Euler scheme (\ref{eq:sEuler}) has complexity of
order $\varepsilon^{-6},$ while the projected one (\ref{eq: EulerPartProj}) has complexity of
order $\varepsilon^{-(4s-6)/(s-2)}$ which is significantly smaller when $s>3.$
Moreover, in \cite{antonelli2002rate} conditions are formulated, guaranteeing
that all measures $\mu_{t},$ $t\geq0,$ possess infinitely smooth exponentially
decaying densities. In this case we can additionally profit from the decay of
the generalized Fourier coefficients $(\gamma_{k})$ such that the convergence
rates in \eqref{eq: main_bound} give rise to a proportional reduction of
computational cost approaching $N^{-1},$ corresponding to a complexity of
order $\varepsilon^{-4}$ (modulo some logarithmic term) for the method (\ref{eq: EulerPartProj}).
\end{rem}

\subsection{Density estimation}

\label{sec:dens} Let us now discuss the estimation of the densities $\mu_{t},$
$t\geq0.$ Let us assume that the formal relationship (\ref{formu}) holds in
the sense that%
\[
\frac{\mu_{s}}{w}=\sum_{k=0}^{\infty}\gamma_{k}(s)\varphi_{k}%
\]
in $L_{2}\left(  \mathbb{R}^{d},w\right)  ,$ i.e. $\mu_{s}^{2}/w\in
L_{1}\left(  \mathbb{R}^{d}\right)  .$ Fix some $t>0,$ $K_{\text{test}}%
\in\mathbb{N}$ and set
\[
\widehat{\mu}_{t}^{K_{\text{test}},K,N}(x):=\sum_{k=1}^{K_{\text{test}}}%
\gamma_{k}^{N}(t)\varphi_{k}(x)w(x)
\]
with $\gamma_{k}^{N}(t):=\frac{1}{N}\sum_{i=1}^{N}\varphi_{k}(X_{t}^{i,K,N}),$
$k=1,\ldots,K_{\text{test}}.$ We obviously have
\begin{align*}
\mathrm{E}\int|\widehat{\mu}_{t}^{K_{\text{test}},K,N}(x)-\mu_{t}%
(x)|^{2}w^{-1}(x)\,dx  &  =\sum_{k=1}^{K_{\text{test}}}\mathrm{E}\left[
|\gamma_{k}^{N}(t)-\gamma_{k}(t)|^{2}\right] \\
&  +\sum_{k=K_{\text{test}}+1}^{\infty}|\gamma_{k}(t)|^{2},
\end{align*}
where (due to (AF))
\begin{align*}
\mathrm{E}\left[  |\gamma_{k}^{N}(t)-\gamma_{k}(t)|^{2}\right]   &
=\mathrm{E}\left[  \left\vert \frac{1}{N}\sum_{j=1}^{N}\varphi_{k}%
(X_{t}^{j,K,N})-\mathrm{E}\left[  \varphi_{k}(X_{t}^{\cdot})\right]
\right\vert ^{2}\right] \\
&  \leq2\mathrm{E}\left[  \left\vert \frac{1}{N}\sum_{j=1}^{N}\left(
\varphi_{k}(X_{t}^{j,K,N})-\varphi_{k}(X_{t}^{j})\right)  \right\vert
^{2}\right] \\
&  +2\mathrm{E}\left[  \left\vert \frac{1}{N}\sum_{j=1}^{N}\left(  \varphi
_{k}(X_{t}^{j})-\mathrm{E}\left[  \varphi_{k}(X_{t}^{j})\right]  \right)
\right\vert ^{2}\right] \\
&  \leq2L_{k,\varphi}^{2}\mathrm{E}\left[  \left\vert X_{t}^{\cdot,K,N}%
-X_{t}^{\cdot}\right\vert ^{2}\right]  +\frac{2}{N}\mathrm{Var}\left[
\varphi_{k}(X_{t})\right]  ,
\end{align*}
since the $X^{j}$ are independent. Theorem~\ref{thmc} now implies
\begin{multline}
\label{eq:dens_error}\left(  \mathrm{E}\int|\widehat{\mu}_{t}^{K_{\text{test}%
},K,N}(x)-\mu_{t}(x)|^{2}\,w^{-1}(x)\,dx\right)  ^{1/2}\lesssim\left(
\frac{1}{N}\sum_{k=1}^{K_{\text{test}}}(L_{k,\varphi}^{2}+D_{k,\varphi}%
^{2})\right)  ^{1/2}\\
+\left(  \sum_{k=1}^{K_{\text{test}}}L_{k,\varphi}^{2}\right)  ^{1/2}\left(
\sum_{k=K+1}^{\infty}(A_{k,\alpha}+A_{k,\beta})\,\Vert\gamma_{k}\Vert
_{L_{p}[0,T]}\right) \\
+\left(  \sum_{k=K_{\text{test}}+1}^{\infty}|\gamma_{k}(t)|^{2}\right)
^{1/2}.
\end{multline}
The last term always converges to zero as $K_{\text{test}}\rightarrow\infty,$
since $\mu_{t}/w\in L_{2}(\mathbb{R}^{d},w).$ The first term can be controlled
for any fixed $K_{\text{test}}$ by taking $N$ large enough. Finally, for any
fixed $K_{\text{test}},$ the second term can be made small by taking $K$ large
enough and using the condition (AC).

\section{Specific models}

\label{sec:affine}

\subsection{Generalized Shimizu-Yamada Models}

\label{expls}

Inspired by the work of Shimizu and Yamada \cite{shimizu1972phenomenological},
\cite{zhang1997numerical} and \cite{lo2012simple}, we consider one-dimensional
McKean-Vlasov equations of the form \eqref{eq:sde} with
\[
a(x,u):=a^{0}(u)+a^{1}(u)x,\quad b(x,u):=b(u).
\]
This class of models allows for a linear dependence of drift on the
distribution of $X$ through $\mathrm{E}\left[  a^{0}(X_{t})\right]  $ and
$\mathrm{E}\left[  a^{1}(X_{t})\right]  .$ Let us define for polynomially
bounded and measurable functions $a^{j}$ and $b$ the generalized Gauss
transforms,
\begin{align*}
H_{a^{j}}(p,q)  &  :=\frac{1}{\sqrt{2\pi q}}\int a^{j}(u)e^{-\frac{(p-u)^{2}%
}{2q}}du,\quad j=0,1,\\
H_{b}(p,q)  &  :=\frac{1}{\sqrt{2\pi q}}\int b(u)e^{-\frac{(p-u)^{2}}{2q}%
}du,\text{ \ \ }p\in\mathbb{R},\text{ \ }q>0.
\end{align*}
Let moreover $a^{j}$ and $b$ be such that the partial derivatives,%
\begin{gather}
\partial_{p}H_{a^{j}}(p,q),\,\partial_{q}H_{a^{j}}(p,q)\text{ \ \ }%
j=0,1,\text{ \ \ and \ \ }\partial_{p}H_{b}(p,q),\,\partial_{q}H_{b}%
(p,q),\text{\ }\nonumber\\
\text{extend continuously to any }(p,q)\in\mathbb{R\times R}_{\geq0}.\text{\ }
\label{rder0}%
\end{gather}
It is not difficult to see that (\ref{rder0}) holds if $a^{j}$ and $b$ are
entire functions for which the coefficients of their power series around $u=0$
decay fast enough to zero (which is trivially satisfied for any polynomial). A
complete characterization of $a^{j}$ and $b$ such that (\ref{rder0}) holds, is
connected with analytic vectors for semigroups related to the heat kernel and
considered beyond the scope of this paper however.

\begin{thm}
\label{thm:affine_expl} Let $a^{j}$ and $b$ satisfy (\ref{rder0}). (i) Then
the following system of ODEs
\begin{align}
G_{t}^{\prime}  &  =H_{b}^{2}\left(  A_{t},G_{t}\right)  +2H_{a^{1}}\left(
A_{t},G_{t}\right)  G_{t}\label{ODE}\\
A_{t}^{\prime}  &  =H_{a^{0}}\left(  A_{t},G_{t}\right)  +H_{a^{1}}\left(
A_{t},G_{t}\right)  A_{t},\text{ \ \ }(A_{0},G_{0})=\left(  x_{0},0\right)
,\nonumber
\end{align}
has for $0\leq t<t_{\infty}\leq\infty,$ i.e. up to some possibly finite
exploding time $t_{\infty},$ a unique solution $(A_{t},G_{t})\in
\mathbb{R\times R}_{\geq0}.$ (ii) The McKean-Vlasov SDE
\begin{equation}
dX_{t}=(\mathrm{E}\left[  a^{0}(X_{t})\right]  +X_{t}\,\mathrm{E}\left[
a^{1}(X_{t})\right]  )\,dt+\mathrm{E}\left[  b(X_{t})\right]  dW_{t},\quad
X_{0}=x_{0} \label{MVl}%
\end{equation}
is then equivalent to
\begin{equation}
dX_{t}=\left(  H_{a^{0}}\left(  A_{t},G_{t}\right)  +H_{a^{1}}\left(
A_{t},G_{t}\right)  X_{t}\right)  dt+H_{b}\left(  A_{t},G_{t}\right)
dW_{t},\text{ \ \ }X_{0}=x_{0}, \label{exp0}%
\end{equation}
and has explicit solution,
\begin{align}
X_{t}  &  =x_{0}e^{\int_{0}^{t}H_{a^{1}}\left(  A_{s},G_{s}\right)  ds}%
+\int_{0}^{t}H_{a^{0}}\left(  A_{s},G_{s}\right)  e^{\int_{s}^{t}H_{a^{1}%
}\left(  A_{r},G_{r}\right)  dr}ds\label{exp}\\
&  +\int_{0}^{t}H_{b}\left(  A_{s},G_{s}\right)  e^{\int_{s}^{t}H_{a^{1}%
}\left(  A_{r},G_{r}\right)  dr}dW_{s},\text{ \ \ }0\leq t<t_{\infty}%
\leq\infty.\nonumber
\end{align}
Note: the Wiener integral in (\ref{exp}) can be interpreted by an ordinary
integral after partial integration, due to the smoothness of the
(deterministic) integrand.
\end{thm}

\subsection{Affine structures}

\label{affineex} Let us consider affine functions
\begin{align*}
a^{0}(u)  &  =a_{0}^{0}+a_{1}^{0}u,\\
a^{1}(u)  &  =a_{0}^{1}+a_{1}^{1}u,\\
b(u)  &  =b_{0}+b_{1}u.
\end{align*}
Then for $c\equiv a^{0},$~$c\equiv a^{1},$ and $c\equiv b,$ respectively, we
have%
\begin{align*}
H_{c}(p,q)  &  =\frac{1}{\sqrt{2\pi q}}\int c(u)e^{-\frac{(p-u)^{2}}{2q}}du\\
&  =\frac{1}{\sqrt{2\pi q}}\int c_{0}e^{-\frac{(p-u)^{2}}{2q}}du+\frac
{1}{\sqrt{2\pi q}}\int c_{1}ue^{-\frac{(p-u)^{2}}{2q}}du\\
&  =c_{0}+c_{1}p
\end{align*}
with $c(u)=c_{0}+c_{1}u.$ In particular, the $H_{c}(p,q)$ do not depend on
$q,$ and so (\ref{ODE}) simplifies to%
\begin{equation}
A_{t}^{\prime}=a_{0}^{0}+\left(  a_{1}^{0}+a_{0}^{1}\right)  A_{t}+a_{1}%
^{1}A_{t}^{2},\text{ \ \ }A_{0}=x_{0}. \label{ODE1}%
\end{equation}
We first consider the case $a_{1}^{1}=0,$ then (\ref{ODE1}) reads
$A_{t}^{\prime}=a_{0}^{0}+\left(  a_{1}^{0}+a_{0}^{1}\right)  A_{t}$ with
solution
\begin{align}
A_{t}  &  =\left(  x_{0}+\frac{a_{0}^{0}}{a_{1}^{0}+a_{0}^{1}}\right)
e^{\left(  a_{1}^{0}+a_{0}^{1}\right)  t}-\frac{a_{0}^{0}}{a_{1}^{0}+a_{0}%
^{1}}\text{ \ \ if \ \ }a_{1}^{0}+a_{0}^{1}\neq0,\text{ \ \ and}\label{a11}\\
A_{t}  &  =x_{0}+a_{0}^{0}t\text{ \ \ if \ \ }a_{1}^{0}+a_{0}^{1}=0.\nonumber
\end{align}
For the case $a_{1}^{1}\neq0$ the solution (checked by Mathematica) is as
follows. If $D:=\left(  a_{1}^{0}+a_{0}^{1}\right)  ^{2}-4a_{0}^{0}a_{1}%
^{1}<0,$ $a_{1}^{1}\neq0,$
\begin{equation}
A_{t}=-\frac{\left(  a_{1}^{0}+a_{0}^{1}\right)  }{2a_{1}^{1}}+\frac{\sqrt
{-D}}{2a_{1}^{1}}\tan\left[  \ \frac{1}{2}\sqrt{-D}t+\arctan\left[
\frac{a_{1}^{0}+a_{0}^{1}+2a_{1}^{1}x_{0}}{\sqrt{-D}}\right]  \right]  .
\label{dl0}%
\end{equation}
If $D>0,$ $a_{1}^{1}\neq0,$
\begin{equation}
A_{t}=\frac{1}{2}\left(  \sqrt{D}-a_{1}^{0}-a_{0}^{1}\right)  /a_{1}^{1}%
+\frac{x_{0}-\frac{1}{2}\left(  \sqrt{D}-a_{1}^{0}-a_{0}^{1}\right)
/a_{1}^{1}}{1+\frac{1}{2}\left(  \sqrt{D}+a_{1}^{0}+a_{0}^{1}+2a_{1}^{1}%
x_{0}\right)  \left(  e^{\ -\sqrt{D}t}-1\right)  /\sqrt{D}}. \label{dg0}%
\end{equation}
If $D=0,$ $a_{1}^{1}\neq0,$
\begin{equation}
A_{t}=-\frac{a_{1}^{0}+a_{0}^{1}}{2a_{1}^{1}}+\frac{1}{a_{1}^{1}}\frac
{a_{1}^{0}+a_{0}^{1}+2a_{1}^{1}x_{0}}{2-\left(  a_{1}^{0}+a_{0}^{1}+2a_{1}%
^{1}x_{0}\right)  t}. \label{D0}%
\end{equation}
As a result, the McKean-Vlasov SDE
\[
dX_{t}=(a_{0}^{0}+a_{1}^{0}A_{t}+\left(  a_{0}^{1}+a_{1}^{1}A_{t}\right)
X_{t})dt+\left(  b_{0}+b_{1}A_{t}\right)  dW_{t}%
\]
has the following (unique) solution
\begin{align}
X_{t}  &  =x_{0}e^{\int_{0}^{t}\left(  a_{0}^{1}+a_{1}^{1}A_{s}\right)
ds}+\int_{0}^{t}\left(  a_{0}^{0}+a_{1}^{0}A_{s}\right)  e^{\int_{s}%
^{t}\left(  a_{0}^{1}+a_{1}^{1}A_{r}\right)  dr}ds\nonumber\\
&  +\int_{0}^{t}\left(  b_{0}+b_{1}A_{s}\right)  e^{\int_{s}^{t}\left(
a_{0}^{1}+a_{1}^{1}A_{r}\right)  dr}dW_{s}, \label{ex}%
\end{align}
where $A_{t}$ is given by (\ref{a11}), (\ref{dl0}), (\ref{dg0}), or (\ref{D0}).

\begin{example}
By taking in Section~\ref{affineex}%
\[
a(x,u)=a_{1}^{0}u+a_{0}^{1}x,\text{ \ \ }b(x,u)=b_{0},\text{ \ \ }a_{1}%
^{0}+a_{0}^{1}<0,
\]
we get essentially the Shimizu-Yamada model. From (\ref{a11}) we then have%
\[
A_{t}=x_{0}e^{\ \left(  a_{1}^{0}+a_{0}^{1}\right)  t},
\]
and from (\ref{ex}) we then get the explicit solution%
\[
X_{t}=x_{0}e^{\left(  a_{1}^{0}+a_{0}^{1}\right)  t}+\int_{0}^{t}b_{0}%
e^{a_{0}^{1}(t-s)}dW_{s}%
\]
which is Gaussian with mean $x_{0}e^{\left(  a_{1}^{0}+a_{0}^{1}\right)  t}$
and variance $b_{0}^{2}\frac{e^{2a_{0}^{1}t}-1}{2a_{0}^{1}},$ and which is
consistent with the terminology in (\cite{frank2004stochastic}, Section 3.10),
where $a_{1}^{0}+a_{0}^{1}=-\gamma$ and $a_{0}^{1}=-\gamma-\kappa.$
\end{example}

\begin{example}
\label{affineex1} By taking in Section~\ref{affineex}
\[
a(x,u)=\left(  a_{0}^{1}+a_{1}^{1}u\right)  x,\text{ \ \ }b(x,u)=b_{0},
\]
we straightforwardly get from (\ref{dg0}),
\begin{equation}
A_{t}=\frac{x_{0}e^{\ a_{0}^{1}t}}{1-\frac{a_{1}^{1}}{a_{0}^{1}}x_{0}\left(
e^{\ a_{0}^{1}t}-1\right)  }, \label{at}%
\end{equation}
and
\begin{equation}
X_{t}=x_{0}e^{\int_{0}^{t}\left(  a_{0}^{1}+a_{1}^{1}A_{s}\right)  ds}%
+\int_{0}^{t}b_{0}e^{\int_{s}^{t}\left(  a_{0}^{1}+a_{1}^{1}A_{r}\right)
dr}dW_{s}, \label{exs}%
\end{equation}
respectively. Plugging (\ref{at}) into (\ref{exs}) then yields
\[
X_{t}=\frac{x_{0}e^{\ a_{0}^{1}t}}{1-\frac{a_{1}^{1}}{a_{0}^{1}}x_{0}\left(
e^{\ a_{0}^{1}t}-1\right)  }+\frac{b_{0}e^{\ a_{0}^{1}t}}{1-\frac{a_{1}^{1}%
}{a_{0}^{1}}x_{0}\left(  e^{\ a_{0}^{1}t}-1\right)  }\Gamma_{t}%
\]
with Gaussian $\Gamma_{t}=\int_{0}^{t}\left(  1-\frac{a_{1}^{1}}{a_{0}^{1}%
}x_{0}\left(  e^{\ a_{0}^{1}s}-1\right)  \right)  e^{\ -a_{0}^{1}s}dW_{s}.$ In
particular, if $a_{0}^{1}=0$ we get%
\[
A_{t}=\frac{x_{0}}{1-a_{1}^{1}x_{0}t},
\]
and
\[
X_{t}=\frac{x_{0}}{1-a_{1}^{1}x_{0}t}+b_{0}\int_{0}^{t}\frac{1-a_{1}^{1}%
x_{0}s}{1-a_{1}^{1}x_{0}t}dW_{s}.
\]

\end{example}

\begin{rem}
\label{nonexp} From Example~\ref{affineex1} it is clear that if $a_{1}^{1}%
\neq0$, the affine McKean-Vlasov solution may explode in finite time. This is
not surprising since in this case the derivative $\partial_{u}a(x,u)$\ is
unbounded and so the main results in \cite{antonelli2002rate} do not apply. On
the other hand, it is easy to check that for the case $a_{1}^{1}\neq0,$ the
affine solutions in Section~\ref{affineex} are non-exploding whenever,
\[
D\geq0\text{ and }\sqrt{D}\geq a_{1}^{0}+a_{0}^{1}+2a_{1}^{1}x_{0}.
\]
That is, in the case $D\geq0,$ $a_{1}^{1}\neq0,$ it is always possible to
choose $x_{0}$ such that the solution does or does not explode.
\end{rem}

\subsection{Kuramoto-Shinomoto-Sakaguchi type models}

In the Kuramoto-Shinomoto-Sakaguchi model  the nonlinear one-dimensional
Fokker-Planck equation (\ref{eq:FP}) is considered in the domain $(t,x)\in\left(
0,\infty\right)  \times\left(  0,2\pi\right),$ where $b=1,$ \(a(x,y)  =a(x-y)=-\frac{d}{dx}U_{MF}(x-y)\) with
\begin{align*}
U_{MF}(z)  =-\sum_{n=1}^{\infty}c_{n}\cos(nz)
\end{align*}
and the process starts in \(u\) at time zero, for some fixed $u\in\left(
0,2\pi\right),$ see for details \cite{frank2004stochastic} (Sect. 5.3.2).
Thus $a$ is a $2\pi$-periodic function related to a $2\pi$-periodic potential.
Let us  consider the corresponding McKean-Vlasov SDE %
\begin{equation}
\left\{
\begin{array}
[c]{ll}%
X_{t} & =u+\int_{0}^{t}\int_{\mathbb{R}}a(X_{s}-y)\mu_{s}(dy)ds+W_{t}\\
\mu_{t} & =\mathrm{Law}(X_{t}),\quad t\in\lbrack0,T],
\end{array}
\right.  \label{eq:sde1}%
\end{equation}
and define the integer valued function $k(x):=\max\left\{  j\in\mathbb{Z}:2\pi j\leq
x\right\}.$ Obviously, the process%
\begin{equation}
Y_{t}:=X_{t}-2\pi k(X_{t})\label{dy}%
\end{equation}
has state space $\left[  0,2\pi\right).$ Let $\rho_{t}(x;u)=\rho_{t}(x)$ be
the density of $Y_{t},$ which is concentrated on $\left(  0,2\pi\right).$
Note that for any $2\pi$-periodic function $f$ we have by (\ref{dy}) that%
\[
\int_{0}^{2\pi}f(x)\rho_{t}(x)dx=\mathrm{E}\left[  f(Y_{t})\right]  =\mathrm{E}\left[
f(X_{t})\right]  =\int_{-\infty}^{\infty}f(x)\mu_{t}(x)dx,
\]
and for any test function $g$ with support in $(0,2\pi)$ it holds that,%
\begin{align*}
\int_{0}^{2\pi}g(x)\rho_{t}(x)dx &  =\mathrm{E}\left[  g(Y_{t})\right]  =\mathrm{E}\left[
g(X_{t}-2\pi k(X_{t}))\right]  \\
&  =\sum_{j\in\mathbb{Z}}\int_{2\pi j}^{2\pi\left(  j+1\right)  }g(x-2\pi
j)\mu_{t}(x)dx\\
&  =\int_{0}^{2\pi}g(z)\sum_{j\in\mathbb{Z}}\mu_{t}(z+2\pi j)dz,
\end{align*}
that is%
\begin{equation}
\rho_{t}(z)=\sum_{j\in\mathbb{Z}}\mu_{t}(z+2\pi j)\label{to}%
\end{equation}
for \(z\in (0,2\pi).\)
Thus, in particular,%
\begin{equation}
\int_{\mathbb{R}}a(x-y)\mu_{t}(y)dy=\int_{0}^{2\pi}a(x-y)\rho_{t}%
(y)dy.\label{on}%
\end{equation}
an  (\ref{eq:sde1}) is equivalent to%
\begin{align*}
X_{t} &  =u+\int_{0}^{t}\int_{0}^{2\pi}a(X_{s}-y)\rho_{s}(y)dyds+W_{t}\\
\rho_{t} &  =\mathrm{Law}(Y_{t}),\quad t\in\lbrack0,T],
\end{align*}
(see (\ref{dy})). Note that by using (\ref{to}) and (\ref{on}) it
straightforwardly follows that $\rho_{t}(x)=\rho_{t}(x;u)$ satisfies
(\ref{eq:FP}) in the above context. Instead of taking the scalar product in
$L_{2}\left(  \mathbb{R}^{d},w\right),$ we now consider the scalar product in
$L_{2}\left(  \left[  0,2\pi\right)  \right),$ i.e. $w\equiv1,$ and take for
$\left(  \varphi_{k}\right)  $ the standard (total) orthonormal trigonometric
basis consisting of the $2\pi$-periodic functions $\left(  2\pi\right)
^{-1/2},$ $\pi^{-1/2}\cos\left(  my\right)  $ and $\pi^{-1/2}\sin\left(
my\right)  ,$ $m=1,2,\ldots $ suitably ordered. Thus, by defining%
\begin{align*}
\gamma_{k}(t)   =\int_{0}^{2\pi}\rho_{t}(y)\varphi_{k}\left(  y\right)
dy, \quad
\alpha_{k}\left(  x\right)     =\int_{0}^{2\pi}a(x-y)\varphi_{k}\left(
y\right)  dy,
\end{align*}
one has%
\[
\text{ }\rho_{t}(y)=\sum_{k=0}^{\infty}\gamma_{k}(t)\varphi_{k}\left(
y\right)  \text{\ \ \ and }\int_{\mathbb{R}}a(x-y)\mu_{t}(y)dy=\sum
_{k=0}^{\infty}\alpha_{k}\left(  x\right)  \gamma_{k}(t),
\]
due to (\ref{on}). That is (\ref{eq:par_proj}) reads,%
\[
X_{t}^{i,K,N}=u+\int_{0}^{t}\sum_{k=0}^{K}\gamma_{k}^{N}(s)\alpha_{k}%
(X_{s}^{i,K,N})\,ds+\,W_{t}^{i}%
\]
with $\gamma_{k}^{N}$ as in (\ref{heu}). Next we may follow
(\ref{eq: EulerPartProj}) for the corresponding Euler scheme. Finally, the
estimator for the density $\rho_{t}$ reads%
\begin{equation}
\widehat{\rho}_{t}^{K_{\text{test}},K,N}(y):=\sum_{k=1}^{K_{\text{test}}%
}\gamma_{k}^{N}(t)\varphi_{k}(y)\label{den}%
\end{equation}
(cf. the estimator for $\mu_{t}$ in Section 3.1).

\section{Numerical test cases}

\label{sec:num}

\subsection{Affine MVSDE models}

Let us now test the numerical performance of the projected particle approach
for the processes discussed in Section~\ref{expls}. Consider the situation
where
\[
a^{0}(u)=(1+u^{M})\exp(-u^{2}/2),\quad a^{1}(u)=\rho\exp(-u^{2}/2),\quad
b(x,u)\equiv\sigma
\]
for some $M>0,$ $\rho\geq0$ and $\sigma>0.$ Then, by using the Hermite
functions \eqref{eq:herm_func} with $w\equiv1$ and the well-known identity%
\[
u^{M}=\frac{1}{2^{M}}\sum_{m=0}^{\left\lfloor \frac{M}{2}\right\rfloor }%
\frac{M!}{m!(M-2m)!}H_{M-2m}(u),
\]
we derive straightforwardly,
\begin{align*}
\alpha_{k}(x)  &  =\int(1+u^{M})\exp(-u^{2})\overline{H}_{k}(u)\,du+\rho
x\cdot\int\exp(-u^{2})\overline{H}_{k}(u)\,du\\
&  =%
\begin{cases}
0 & \text{if \ \ }k>M \text{ or } k \text{ is uneven }\\
\frac{\pi^{1/4}}{2^{M-k/2}}\frac{M!}{\left(  \frac{M-k}{2}\right)  !\sqrt{k!}}
& \text{if \ \ }0\leq k\leq M,\text{ \ s.t. \ }M-k\text{ \ is even}%
\end{cases}
\\
&  +\left(  1+\rho x\right)  \cdot%
\begin{cases}
\pi^{1/4} & \text{if \ \ }k=0,\\
0 & \text{if \ \ }k>0.
\end{cases}
\end{align*}
On the other hand, by some algebra we get
\begin{align*}
H_{a^{0}}(p,q)  &  =\frac{1}{\sqrt{2\pi q}}\int(1+u^{M})e^{-u^{2}/2}%
e^{-\frac{(p-u)^{2}}{2q}}du=\frac{1}{\sqrt{1+q}}e^{-\frac{p^{2}}{2(1+q)}}\\
&  +\frac{1}{\sqrt{2\pi(1+q)}}e^{-\frac{p^{2}}{2(1+q)}}\int\left(  \sqrt
{\frac{q}{1+q}}y+\frac{p}{1+q}\right)  ^{M}e^{-y^{2}/2}dy,
\end{align*}
and
\[
H_{a^{1}}(p,q)=\frac{\rho}{\sqrt{1+q}}e^{-\frac{p^{2}}{2(1+q)}}.
\]
The explicit solution of the MVSDE
\begin{equation}
dX_{t}=\left(  \mathsf{E}\left[  (1+X_{t}^{M})\exp(-X_{t}^{2}/2)\right]
+\rho\,X_{t}\mathsf{E}\left[  \exp(-X_{t}^{2}/2)\right]  \right)  \,dt+\sigma
dW_{t} \label{eq:xt_affine_test}%
\end{equation}
is given by (\ref{ODE}) and \eqref{exp}. Hence the density of $X_{t}$ is
normal with mean
\[
x_{0}e^{\int_{0}^{t}H_{a^{1}}\left(  A_{s},G_{s}\right)  ds}+\int_{0}%
^{t}H_{a^{0}}\left(  A_{s},G_{s}\right)  e^{\int_{s}^{t}H_{a^{1}}\left(
A_{r},G_{r}\right)  dr}\,ds
\]
and variance
\[
\sigma^{2}\int_{0}^{t}e^{2\int_{s}^{t}H_{a^{1}}\left(  A_{r},G_{r}\right)
dr}\,ds.
\]
In our numerical example we take $M=2,$ $\rho=-1,$ $\sigma=1$ and $x_{0}=0.$
Our aim is to approximate the normal density of $X_{1}$ by using our projected
particle method based on Hermite basis. To this end, we first simulate $N$
paths of the process $\bar{X}^{i,K,N},$ defined in \eqref{eq: EulerPartProj}
with a time step $h=0.02.$ Since $M=2,$ the case $K=2$ corresponds to a
perfect approximation of the integral $\int_{\mathbb{R}^{d}}a(x,y)\mu
_{t}(y)\,dy=\sum_{k=0}^{2}\alpha_{k}(x)\gamma_{k}(t).$ Next using the obtained
sample $\bar{X}_{1}^{1,K,N},\ldots,\bar{X}_{1}^{N,K,N},$ we construct
projection estimates for the density of $X_{1}$ by using Hermite basis
functions of order $K_{\mathrm{test}}\in\{1,2,\ldots,10\}$. The mean ($0.727$)
and the variance ($0.487$) of the true normal density are approximated by
solving the ODE system~ \eqref{ODE} using Euler method with time step
$0.0001.$ The Figure~\ref{fig: affine_density} shows the box plots of $L_{2}%
$-distance between $\mu_{1}$ and $\widehat{\mu}_{t}^{K_{\text{test}},K,N}$ for
$K\in\{1,2\}$ based on $50$ different replications of the process $\bar
{X}^{i,K,N}.$ {As can be seen, the choice of $K_{\mathrm{test}}$ is crucial
and depends on $K$ and $N$. It also should be stressed that the truncation
error dominates the statistical one already for medium sample sizes. An
optimal balance between $K,$ $N$ and $K_{\mathrm{test}}$ can be found by
analyzing the right hand side of \eqref{eq:dens_error} under various
assumptions on the coefficients $(A_{k,\alpha}),$ $(A_{k,\beta})$ and
$(\gamma_{k}$).} \begin{figure}[pt]
\centering
\includegraphics[width=\linewidth]{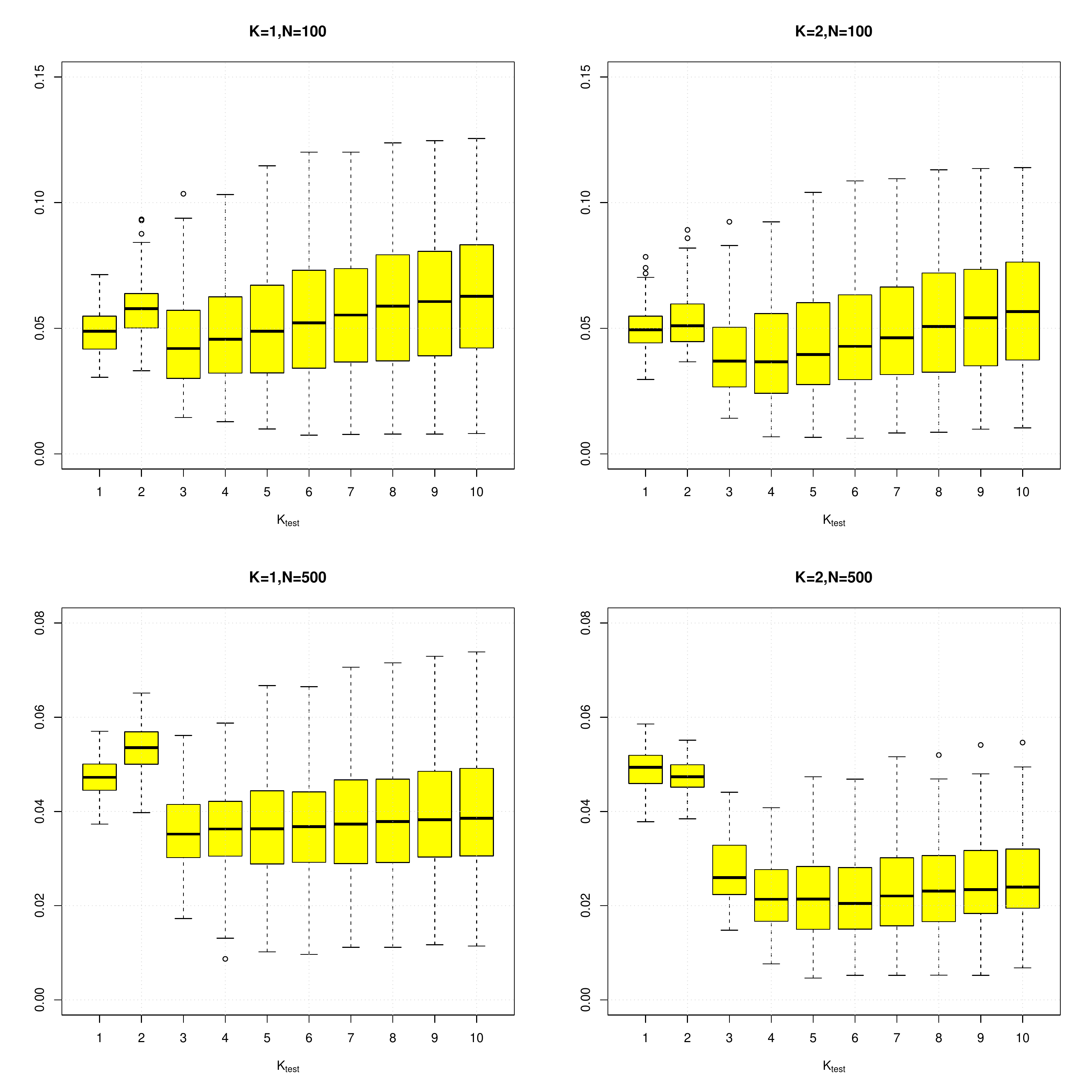} \caption{Box plots
of $L_{2}$-distances between the true normal density of $X_{1}$ with mean
$0.727$ and variance $0.487$ in the model \eqref{eq:xt_affine_test} and its
estimates obtained by using $N\in\{100,500\}$ paths, Hermite basis functions
up to order $K\in\{1,2\}$ to approximate coefficients $\alpha_{k}$ and Hermite
basis functions up to order $K_{\mathrm{test}}\in\{1,2,\ldots,10\}$ for
projection based density estimation. }%
\label{fig: affine_density}%
\end{figure}

\subsection{Convolution-type MVSDE models}

Consider the MVSDE of the form:
\[
dX_{t}=\E_{X^{\prime}}\left[  Q(X_{t}-X_{t}^{\prime})\right]  \,dt+\sigma
\,dW_{t},\quad t\in\lbrack0,1],\quad X_{0}\sim\mathcal{N}(0,1),
\]
i.e. of the form \eqref{eq:sde} with $a(x,y)=Q(x-y),$ $b(x,y)=\sigma$ and
$\mu_{0}(x)=(1/\sqrt{2\pi})e^{-x^{2}/2}.$ Let us again use the Hermite basis
to approximate the density of $X_{t}$ for any $t\in\lbrack0,1].$ In the case
$Q(x)=e^{-x^{2}/2},$ we explicitly derive {via repeated integration by parts}
\begin{align*}
\int_{\mathbb{R}}e^{-(x-y)^{2}/2-x^{2}/2}H_{n}(x)\,dx  &  =\frac{e^{-y^{2}/4}%
}{2}\int e^{-(z-y)^{2}/4}H_{n}(z/2)\,dz\\
&  =\sqrt{\pi}\,\frac{e^{-y^{2}/4}}{2}\left(  \frac{1}{2}\right)
^{n-1}(2y)^{n}.
\end{align*}
As a result
\[
\alpha_{n}(y)=\int e^{-(x-y)^{2}/2-x^{2}/2}\overline{H}_{n}(x)\,dx=\pi
^{1/4}\left(  \frac{1}{2}\right)  ^{n/2}\frac{y^{n}}{\sqrt{n!}}e^{-y^{2}/4},
\]
where $\overline{H}_{n}$ stands for the normalized Hermite polynomial of order
$n.$ We take $\sigma=0.1.$ Using the Euler scheme \eqref{eq:sEuler} with time
step $h=1/L=0.01$, we first simulate $N=500$ paths of the time discretized
process $\bar{X}^{\cdot,N}.$ Next, by means of the closed form expressions for
$\alpha_{n},$ we generate $N$ paths of the projected approximating process
$\bar{X}^{\cdot,K,N},$ $K\in\{1,\ldots,20\}$ {(see (\ref{eq: EulerPartProj}%
))}, using the same Wiener increments as for $\bar{X}^{\cdot,N},$ so that the
approximations $\bar{X}^{\cdot,N}$ and $\bar{X}^{\cdot,K,N}$ are coupled.
Finally, we compute the strong approximation error
\[
E_{N,K}=\sqrt{\frac{1}{N}\sum_{i=1}^{N}\bigl(\bar{X}_{1}^{i,K,N}-\bar{X}%
_{1}^{i,N}\bigr)^{2}}%
\]
{of the projective system relative to the system \eqref{eq:sEuler}} and record
times needed to compute approximations $\bar{X}_{1}^{\cdot,N}$ and $\bar
{X}_{1}^{\cdot,K,N},$ respectively. Figure~\ref{fig: inter_res} shows the
{(natural)} logarithm of $E_{N,K}$ versus the logarithm of the corresponding
{(relative) computational time gain defined as (comp. time due to
\eqref{eq:sEuler} $-$ comp. time due to (\ref{eq: EulerPartProj}))/comp. time
due to \eqref{eq:sEuler},} for values $K\in\{1,\ldots,20\}$. As can be seen,
the relation between logarithmic strong error and logarithmic computational
time gain can be well approximated by a linear function. On the right-hand
side of Figure~\ref{fig: inter_res} we depict the projection estimate for the
density of $X_{1}$ corresponding to $K=20.$ {Note that we compare two particle
systems (projected and non projected ones) for a fixed $N$ and are mainly
interested in the dependence of {their strong distance} on $K.$ In fact, the
choice of $N$ doesn't have much influence on $E_{N,K},$ provided $N$ is large
enough.}

\begin{figure}[pt]
\centering
\includegraphics[width=\linewidth]{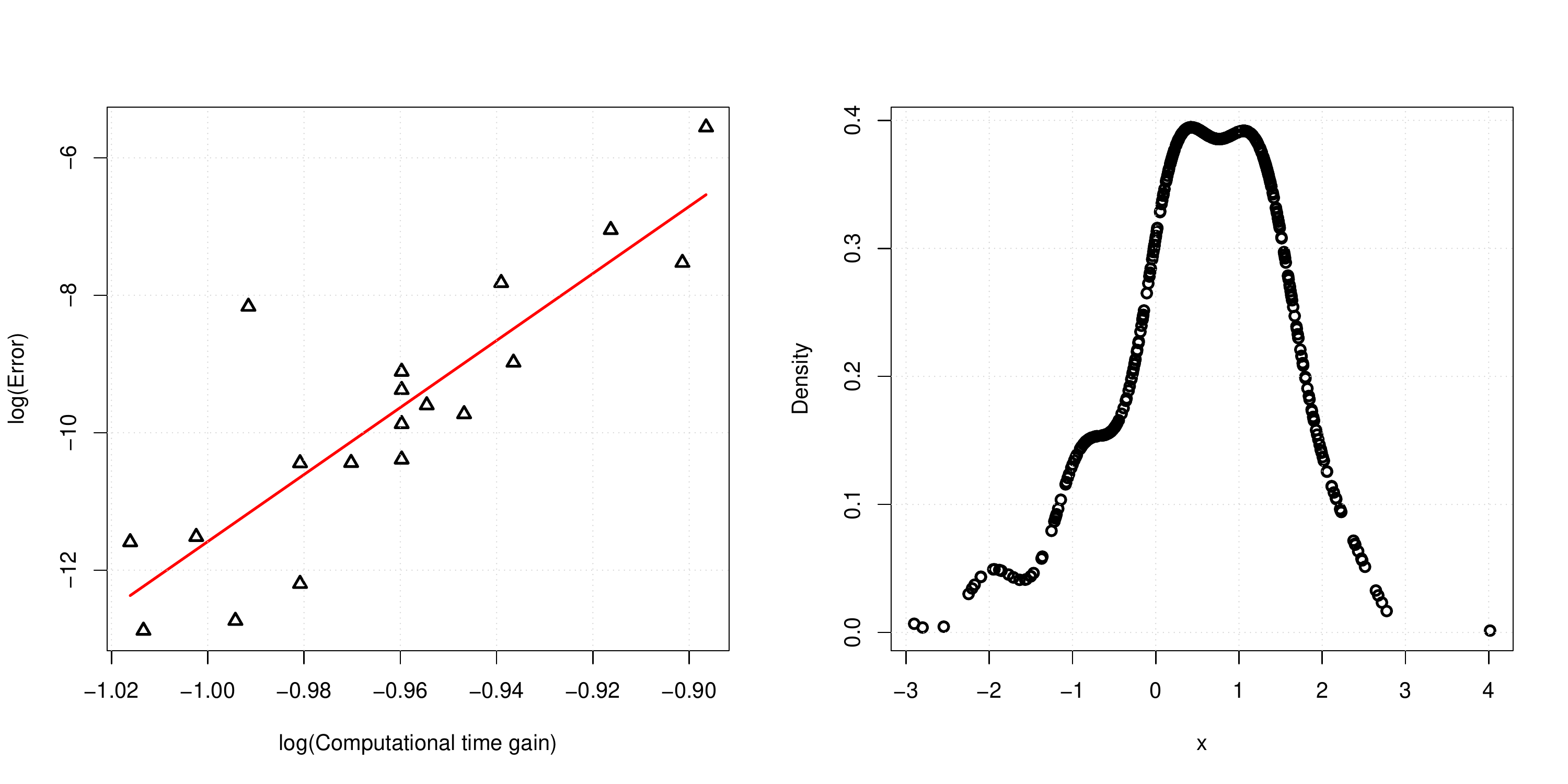} \caption{Left:
Strong error $E_{500,K}$ between the solution of projected (see
\eqref{eq: EulerPartProj}) and non-projected (see \eqref{eq:sEuler})
time-discretized particle systems versus the difference (gain) in
computational time. Right: Estimated density of $X_{1}$ using $21$ basis
functions. }%
\label{fig: inter_res}%
\end{figure}

\section{Proofs}

\label{sec:proofs}

\subsection{Proof of Theorem~\ref{thmc}}

Let us introduce
\[
\mathfrak{a}_{K,N}(x,y):=\frac{1}{N}\sum_{j=1}^{N}\sum_{k=1}^{K}\alpha
_{k}(x)\varphi_{k}(y^{j})=\frac{1}{N}\sum_{j=1}^{N}\sum_{k=1}^{K}\varphi
_{k}(y^{j})\int a(x,u)\varphi_{k}(u)w(u)du,
\]%
\[
\mathfrak{b}_{K,N}(x,y):=\frac{1}{N}\sum_{j=1}^{N}\sum_{k=1}^{K}\beta
_{k}(x)\varphi_{k}(y^{j})=\frac{1}{N}\sum_{j=1}^{N}\sum_{k=1}^{K}\varphi
_{k}(y^{j})\int b(x,u)\varphi_{k}(u)w(u)du,
\]
and%
\begin{align*}
\mathfrak{a}_{s}(x)  &  :=\int_{\mathbb{R}^{d}}a(x,u)\mu_{s}(du)ds,\\
\mathfrak{b}_{s}(x)  &  :=\int_{\mathbb{R}^{d}}b(x,u)\mu_{s}(du)ds
\end{align*}
for any $x\in\mathbb{R}^{d},\text{ }y\in\mathbb{R}^{d\times N}.$ We so have
that%
\begin{align*}
\Delta_{t}^{i}:=X_{t}^{i,K,N}-X_{t}^{i}  &  =\int_{0}^{t}\left(
{\mathfrak{a}_{K,N}}({X}_{s}^{i,K,N},X_{s}^{K,N})\,-\mathfrak{a}_{s}(X_{s}%
^{i})\right)  ds\\
&  +\int_{0}^{t}\left(  \mathfrak{b}_{K,N}({X}_{s}^{i,K,N},X_{s}%
^{K,N})\,-\mathfrak{b}_{s}(X_{s}^{i})\right)  dW_{s}^{i},
\end{align*}
where $W^{i},$ $i=1,...,N,$ are i.i.d. copies of the $m$-dimensional Wiener
process $W.$ Hence,
\begin{align}
\left\vert \Delta_{t}^{i}\right\vert ^{p}  &  \leq2^{p-1}t^{p-1}\int_{0}%
^{t}\left\vert {\mathfrak{a}_{K,N}}({X}_{s}^{i,K,N},X_{s}^{K,N}%
)\,-\mathfrak{a}_{s}(X_{s}^{i})\right\vert ^{p}ds\label{ft}\\
&  +2^{p-1}d^{p-1}\sum_{q=1}^{d}\left\vert \int_{0}^{t}\left(  \mathfrak{b}%
_{K,N}^{q}({X}_{s}^{i,K,N},X_{s}^{K,N})\,-\mathfrak{b}_{s}^{q}(X_{s}%
^{i})\right)  dW_{s}^{i}\right\vert ^{p},\nonumber
\end{align}
where%
\begin{equation}
\mathfrak{b}_{K,N}^{q}:=\left(  \mathfrak{b}_{K,N}^{q,1},...,\mathfrak{b}%
_{K,N}^{q,m}\right)  ,\text{ \ \ }q=1,...,d, \label{defbq}%
\end{equation}
denote the rows of the $\mathbb{R}^{d\times m}$ valued $\mathfrak{b}_{K,N},$
and so we have with%
\[
\overline{\Delta_{t}^{p}}:=\frac{1}{N}\sum_{i=1}^{N}\sup_{s\in\lbrack
0,t]}\left\vert \Delta_{s}^{i}\right\vert ^{p}%
\]
the bound
\begin{align}
\overline{\Delta_{t}^{p}}  &  \leq2^{p-1}t^{p-1}\frac{1}{N}\sum_{i=1}^{N}%
\int_{0}^{t}\left\vert {\mathfrak{a}_{K,N}}({X}_{s}^{i,K,N},X_{s}%
^{K,N})\,-\mathfrak{a}_{s}(X_{s}^{i})\right\vert ^{p}ds\nonumber\\
&  +2^{p-1}d^{p-1}\sum_{q=1}^{d}\frac{1}{N}\sum_{i=1}^{N}\sup_{s\in
\lbrack0,t]}\left\vert \int_{0}^{s}\left(  \mathfrak{b}_{K,N}^{q}({X}%
_{s}^{i,K,N},X_{s}^{K,N})\,-\mathfrak{b}_{s}^{q}(X_{s}^{i})\right)  dW_{s}%
^{i}\right\vert ^{p}\nonumber\\
&  =:2^{p-1}t^{p-1}\,\text{Term}_{1}+2^{p-1}d^{p-1}\,\text{Term}_{2}.
\label{huu}%
\end{align}
Assumption (AC) implies%
\begin{align}
\left\vert \mathfrak{a}_{K,N}(x,y)-\mathfrak{a}_{K,N}(x^{\prime},y^{\prime
})\right\vert  &  =\left\vert \frac{1}{N}\sum_{j=1}^{N}\sum_{k=1}^{K}\left(
\alpha_{k}(x)\varphi_{k}(y_{j})-\alpha_{k}(x^{\prime})\varphi_{k}%
(y_{j}^{\prime})\right)  \right\vert \nonumber\\
&  \leq\frac{1}{N}\sum_{j=1}^{N}\sum_{k=1}^{K}\left\vert \alpha_{k}%
(x)-\alpha_{k}(x^{\prime})\right\vert \left\vert \varphi_{k}(y_{j}^{\prime
})\right\vert \nonumber\\
&  +\frac{1}{N}\sum_{j=1}^{N}\sum_{k=1}^{K}\left\vert \alpha_{k}(x)\right\vert
\left\vert \varphi_{k}(y_{j})-\varphi_{k}(y_{j}^{\prime})\right\vert
\nonumber\\
&  \leq\left\vert x-x^{\prime}\right\vert D_{\varphi}B_{\alpha}+\frac
{L_{\varphi}A_{\alpha}}{N}(1+\left\vert x\right\vert )\sum_{j=1}^{N}\left\vert
y_{j}-y_{j}^{\prime}\right\vert . \label{pm0}%
\end{align}
Hence%
\begin{align*}
\left\vert \mathfrak{a}_{K,N}(x,y)-\mathfrak{a}_{K,N}(x^{\prime},y^{\prime
})\right\vert ^{p}  &  \leq2^{p-1}\left\vert x-x^{\prime}\right\vert
^{p}D_{\varphi}^{p}B_{\alpha}^{p}\\
&  +2^{p-1}L_{\varphi}^{p}A_{\alpha}^{p}(1+\left\vert x\right\vert )^{p}%
\frac{1}{N}\sum_{j=1}^{N}\left\vert y_{j}-y_{j}^{\prime}\right\vert ^{p}.
\end{align*}
So it holds that%
\begin{align*}
\left\vert \mathfrak{a}_{K,N}(X_{s}^{i},X_{s})-\mathfrak{a}_{K,N}({X}%
_{s}^{i,K,N},X_{s}^{K,N})\,\right\vert ^{p}  &  \leq2^{p-1}D_{\varphi}%
^{p}B_{\alpha}^{p}|\Delta_{s}^{i}|^{p}\\
&  +2^{p-1}L_{\varphi}^{p}A_{\alpha}^{p}(1+\left\vert X_{s}^{i}\right\vert
)^{p}\frac{1}{N}\sum_{j=1}^{N}|\Delta_{s}^{j}|^{p},
\end{align*}
and then it follows that, with regard to Term$_{1}$,
\begin{align}
\mathrm{E}\left[  \text{Term}_{1}\right]   &  \leq2^{2p-2}D_{\varphi}%
^{p}B_{\alpha}^{p}\int_{0}^{t}\mathrm{E}\left[  \overline{\Delta_{s}^{p}%
}\right]  ds\nonumber\\
&  +2^{2p-2}L_{\varphi}^{p}A_{\alpha}^{p}\int_{0}^{t}\mathrm{E}\left[
\overline{\Delta_{s}^{p}}\cdot\frac{1}{N}\sum_{i=1}^{N}(1+\left\vert X_{s}%
^{i}\right\vert )^{p}\right]  ds\nonumber\\
&  +2^{p-1}\frac{1}{N}\sum_{i=1}^{N}\int_{0}^{t}\mathrm{E}\left[  \left\vert
{\mathfrak{a}_{K,N}}({X}_{s}^{i},X_{s})\,-\mathfrak{a}_{s}(X_{s}%
^{i})\right\vert ^{p}\right]  ds. \label{t11}%
\end{align}
Let us now consider the middle term. Set%
\[
\zeta_{s,N}:=\frac{1}{N}\sum_{i=1}^{N}(1+\left\vert X_{s}^{i}\right\vert
)^{p}-\frac{1}{N}\sum_{i=1}^{N}\mathrm{E}\left[  (1+\left\vert X_{s}%
^{i}\right\vert )^{p}\right]
\]
so that
\begin{align}
\mathrm{E}\left[  \overline{\Delta_{s}^{p}}\cdot\frac{1}{N}\sum_{i=1}%
^{N}(1+\left\vert X_{s}^{i}\right\vert )^{p}\right]   &  =\frac{1}{N}%
\sum_{i=1}^{N}\mathrm{E}\left[  (1+\left\vert X_{s}^{i}\right\vert
)^{p}\right]  \cdot\mathrm{E}\left[  \overline{\Delta_{s}^{p}}\right]
\nonumber\\
&  +\mathrm{E}\left[  \zeta_{s,N}\cdot\overline{\Delta_{s}^{p}}\right]  .
\label{hu1}%
\end{align}
For arbitrary but fixed $\theta>0,$ it holds that%
\begin{equation}
\mathrm{E}\left[  \zeta_{s,N}\cdot\overline{\Delta_{s}^{p}}\right]
=\mathrm{E}\left[  \zeta_{s,N}\cdot\overline{\Delta_{s}^{p}}\,1_{\left\{
\zeta_{s,N}\leq\theta\right\}  }\right]  +\mathrm{E}\left[  \zeta_{s,N}%
\cdot\overline{\Delta_{s}^{p}}\,1_{\left\{  \zeta_{s,N}>\theta\right\}
}\right]  , \label{hu2}%
\end{equation}
where on the one hand%
\begin{equation}
\mathrm{E}\left[  \zeta_{s,N}\cdot\overline{\Delta_{s}^{p}}\,1_{\left\{
\zeta_{s,N}\leq\theta\right\}  }\right]  \leq\theta\mathrm{E}\left[
\overline{\Delta_{s}^{p}}\right]  \label{hu3}%
\end{equation}
and on the other%
\begin{equation}
\mathrm{E}\left[  \zeta_{s,N}\cdot\overline{\Delta_{s}^{p}}\,1_{\left\{
\zeta_{s,N}>\theta\right\}  }\right]  \leq\sqrt{\mathrm{E}\left[  \zeta
_{s,N}^{2}1_{\left\{  \zeta_{s,N}>\theta\right\}  }\right]  }\sqrt
{\mathrm{E}\left[  \left(  \overline{\Delta_{s}^{p}}\right)  ^{2}\right]  }.
\label{hu4}%
\end{equation}
Due to \eqref{eq: xmom} we have that for any $\eta>0,$ there exists
$C_{\theta,\eta}>0$ such that%
\[
\mathrm{E}\left[  \zeta_{s,N}^{2}1_{\left\{  \zeta_{s,N}>\theta\right\}
}\right]  =\frac{1}{N}\mathrm{E}\left[  \left(  \sqrt{N}\zeta_{s,N}\right)
^{2}1_{\left\{  \sqrt{N}\zeta_{s,N}>\theta\sqrt{N}\right\}  }\right]
\leq\frac{C_{\theta,\eta}^{2}}{N^{\eta+1}},\text{ \ \ \ }0\leq s\leq T,
\]
for $N$ large enough and
\begin{align}
\mathrm{E}\left[  \left(  \overline{\Delta_{s}^{p}}\right)  ^{2}\right]   &
\leq\mathrm{E}\left[  \frac{1}{N}\sum_{j=1}^{N}\sup_{r\in\lbrack
0,T]}\left\vert \Delta_{r}^{j}\right\vert ^{2p}\right]  =\mathrm{E}\left[
\sup_{r\in\lbrack0,T]}\left\vert \Delta_{r}^{\cdot}\right\vert ^{2p}\right]
\nonumber\\
&  =\mathrm{E}\left[  \sup_{r\in\lbrack0,T]}\left\vert X_{r}^{\cdot,K,N}%
-X_{r}^{\cdot}\right\vert ^{2p}\right] \nonumber\\
&  \leq2^{2p-1}\mathrm{E}\left[  \sup_{r\in\lbrack0,T]}\left\vert X_{r}%
^{\cdot,K,N}\right\vert ^{2p}\right]  +2^{2p-1}\mathrm{E}\left[  \sup
_{r\in\lbrack0,T]}\left\vert X_{r}^{\cdot}\right\vert ^{2p}\right] \nonumber\\
&  \leq D_{1}+D_{2}=:D^{2}, \label{hu5}%
\end{align}
where due to Theorem~\ref{thm:moments}, Appendix~\ref{app},
\[
2^{2p-1}\mathrm{E}\left[  \sup_{r\in\lbrack0,T]}\left\vert X_{r}^{\cdot
,K,N}\right\vert ^{2p}\right]  \leq D_{1}<\infty\text{ \ \ uniform in }N\text{
and }K,
\]
and
\[
D_{2}:=2^{2p-1}\mathrm{E}\left[  \sup_{r\in\lbrack0,T]}\left\vert X_{r}%
^{\cdot}\right\vert ^{2p}\right]  <\infty
\]
due to (\ref{eq: xmom}). Thus, by combining (\ref{hu1})--(\ref{hu5}), one has
\[
\mathrm{E}\left[  \overline{\Delta_{s}^{p}}\cdot\frac{1}{N}\sum_{i=1}%
^{N}(1+\left\vert X_{s}^{i}\right\vert )^{p}\right]  \leq F_{1}^{p}%
\cdot\mathrm{E}\left[  \overline{\Delta_{s}^{p}}\right]  +\frac{F_{2}%
}{N^{p/2+1/2}}%
\]
with $F_{1}:=\theta^{1/p}+\sup_{0\leq s\leq T}\left\Vert 1+\left\vert
X_{s}\right\vert \right\Vert _{p}$ and $F_{2}:=C_{\theta,p}D,$ where we have
taken $\eta=p.$ Set now
\[
H(s):=\mathrm{E}\left[  \overline{\Delta_{s}^{p}}\right]  ,
\]
then the estimate (\ref{t11}) (cf. (\ref{huu})) reads%
\begin{gather}
\frac{1}{N}\sum_{i=1}^{N}\int_{0}^{t}\mathrm{E}\left[  \left\vert
{\mathfrak{a}_{K,N}}({X}_{s}^{i,K,N},X_{s}^{K,N})\,-\mathfrak{a}_{s}(X_{s}%
^{i})\right\vert ^{p}\right]  ds\nonumber\\
\leq\left(  2^{2p-2}D_{\varphi}^{p}B_{\alpha}^{p}+2^{2p-2}L_{\varphi}%
^{p}A_{\alpha}^{p}F_{1}^{p}\right)  \int_{0}^{t}H(s)ds+2^{2p-2}L_{\varphi}%
^{p}A_{\alpha}^{p}\frac{F_{2}}{N^{p/2+1/2}}t\nonumber\\
+2^{p-1}\frac{1}{N}\sum_{i=1}^{N}\int_{0}^{t}\mathrm{E}\left[  \left\vert
{\mathfrak{a}_{K,N}}({X}_{s}^{i},X_{s})\,-\mathfrak{a}_{s}(X_{s}%
^{i})\right\vert ^{p}\right]  ds. \label{t12}%
\end{gather}
Regarding the term Term$_{2}$ we call upon the Burkholder-Davis-Gundy's
inequality which states that for any $p\geq1,$%
\begin{align*}
&  \left\Vert \sup_{s\in\lbrack0,t]}\left\vert \int_{0}^{s}\left(
\mathfrak{b}_{K,N}^{q}({X}_{s}^{i,K,N},X_{s}^{K,N})\,-\mathfrak{b}_{s}%
^{q}(X_{s}^{i})\right)  dW_{s}^{i}\right\vert \right\Vert _{p}\\
&  \leq C_{p}\left(  \mathrm{E}\left[  \left(  \int_{0}^{t}\left\vert \left(
\mathfrak{b}_{K,N}^{q}({X}_{s}^{i,K,N},X_{s}^{K,N})\,-\mathfrak{b}_{s}%
^{q}(X_{s}^{i})\right)  \right\vert ^{2}ds\right)  ^{p/2}\right]  \right)
^{1/p}.
\end{align*}
This implies that for $p\geq2,$%
\begin{align}
&  \mathrm{E}\sup_{s\in\lbrack0,t]}\left\vert \int_{0}^{s}\left(
\mathfrak{b}_{K,N}^{q}({X}_{s}^{i,K,N},X_{s}^{K,N})\,-\mathfrak{b}_{s}%
^{q}(X_{s}^{i})\right)  dW_{s}^{i}\right\vert ^{p}\label{fg}\\
&  \leq C_{p}^{p}\mathrm{E}\left[  \left(  \int_{0}^{t}\left\vert \left(
\mathfrak{b}_{K,N}^{q}({X}_{s}^{i,K,N},X_{s}^{K,N})\,-\mathfrak{b}_{s}%
^{q}(X_{s}^{i})\right)  \right\vert ^{2}ds\right)  ^{p/2}\right] \nonumber\\
&  \leq C_{p}^{p}t^{p/2-1}\mathrm{E}\left[  \int_{0}^{t}\left\vert \left(
\mathfrak{b}_{K,N}^{q}({X}_{s}^{i,K,N},X_{s}^{K,N})\,-\mathfrak{b}_{s}%
^{q}(X_{s}^{i})\right)  \right\vert ^{p}ds\right] \nonumber\\
&  \leq C_{p}^{p}t^{p/2-1}\mathrm{E}\left[  \int_{0}^{t}\left\vert \left(
\mathfrak{b}_{K,N}({X}_{s}^{i,K,N},X_{s}^{K,N})\,-\mathfrak{b}_{s}(X_{s}%
^{i})\right)  \right\vert ^{p}ds\right]  .\nonumber
\end{align}
Now, completely analogue to the derivation of (\ref{t12}), we get%
\begin{gather}
\frac{1}{N}\sum_{i=1}^{N}\int_{0}^{t}\mathrm{E}\left[  \left\vert
{\mathfrak{b}_{K,N}}({X}_{s}^{i,K,N},X_{s}^{K,N})\,-\mathfrak{b}_{s}(X_{s}%
^{i})\right\vert ^{p}\right]  ds\nonumber\\
\leq\left(  2^{2p-2}D_{\varphi}^{p}B_{\beta}^{p}+2^{2p-2}L_{\varphi}%
^{p}A_{\beta}^{p}F_{1}\right)  \int_{0}^{t}H(s)ds+2^{2p-2}L_{\varphi}%
^{p}A_{\beta}^{p}\frac{F_{2}}{N^{p/2+1/2}}t\nonumber\\
+2^{p-1}\frac{1}{N}\sum_{i=1}^{N}\int_{0}^{t}\mathrm{E}\left[  \left\vert
\mathfrak{b}{_{K,N}}({X}_{s}^{i},X_{s})\,-\mathfrak{b}_{s}(X_{s}%
^{i})\right\vert ^{p}\right]  ds. \label{t21}%
\end{gather}
Now by taking expectations on both sides of (\ref{huu}) and gathering all
together, we arrive at
\begin{gather}
H(t)\leq\left(  D_{\varphi}^{p}B_{\alpha}^{p}T^{p-1}+L_{\varphi}^{p}A_{\alpha
}^{p}F_{1}^{p}T^{p-1}\right. \nonumber\\
\left.  +C_{p}^{p}D_{\varphi}^{p}B_{\beta}^{p}d^{p}T^{p/2-1}+C_{p}%
^{p}L_{\varphi}^{p}A_{\beta}^{p}d^{p}T^{p/2-1}F_{1}^{p}\right)  2^{3p-3}%
\int_{0}^{t}H(s)ds\nonumber\\
+2^{3p-3}\left(  L_{\varphi}^{p}A_{\alpha}^{p}T^{p}+d^{p}C_{p}^{p}L_{\varphi
}^{p}A_{\beta}^{p}T^{p/2}\right)  \frac{F_{2}}{N^{p/2+1/2}}\label{pg}\\
+2^{2p-2}T^{p-1}\frac{1}{N}\sum_{i=1}^{N}\int_{0}^{t}\mathrm{E}\left[
\left\vert {\mathfrak{a}_{K,N}}({X}_{s}^{i},X_{s})\,-\mathfrak{a}_{s}%
(X_{s}^{i})\right\vert ^{p}\right]  ds\nonumber\\
+2^{2p-2}d^{p}C_{p}^{p}T^{p/2-1}\frac{1}{N}\sum_{i=1}^{N}\int_{0}%
^{t}\mathrm{E}\left[  \left\vert \mathfrak{b}{_{K,N}}({X}_{s}^{i}%
,X_{s})\,-\mathfrak{b}_{s}(X_{s}^{i})\right\vert ^{p}\right]  ds.\nonumber
\end{gather}
We next proceed with explicit estimates for the last two terms above. Let us
write
\[
\mathfrak{a}_{K,N}(X_{s}^{i},X_{s})-\mathfrak{a}_{s}(X_{s}^{i})=\sum_{k=1}%
^{K}\alpha_{k}(X_{s}^{i})\sum_{j=1}^{N}\frac{1}{N}\left(  \varphi_{k}%
(X_{s}^{j})-\gamma_{k}(s)\right)  -\sum_{k=K+1}^{\infty}\alpha_{k}(X_{s}%
^{i})\gamma_{k}(s),
\]
then we have by the Minkowski inequality,
\begin{align*}
\left\Vert \mathfrak{a}_{K,N}(X_{s}^{i},X_{s})-\mathfrak{a}_{s}(X_{s}%
^{i})\right\Vert _{p}  &  \leq\sum_{k=1}^{K}\left\Vert \alpha_{k}(X_{s}%
^{i})\frac{1}{N}\sum_{j=1}^{N}\xi_{k}^{j}\right\Vert _{p}\\
&  +\sum_{k=K+1}^{\infty}\left\Vert \alpha_{k}(X_{s}^{i})\gamma_{k}%
(s)\right\Vert _{p},
\end{align*}
where $\xi_{k}^{j}:=\varphi_{k}(X_{s}^{j})-\gamma_{k}(s),$ $j=1,\ldots,N,$
have mean zero. Let us now observe that
\begin{align*}
\mathrm{E}\left[  \left.  \left\vert \sum_{j=1}^{N}\xi_{k}^{j}\right\vert
^{p}\right\vert X^{i}\right]   &  =\mathrm{E}\left[  \left.  \left\vert
\xi_{k}^{i}+\sum_{j\neq i}^{N}\xi_{k}^{j}\right\vert ^{p}\right\vert
X^{i}\right] \\
&  \leq2^{p-1}\mathrm{E}\left[  \left.  \left\vert \xi_{k}^{i}\right\vert
^{p}+\left\vert \sum_{j\neq i}^{N}\xi_{k}^{j}\right\vert ^{p}\right\vert
X^{i}\right] \\
&  \leq2^{2p-1}D_{k,\varphi}^{p}+2^{p-1}\mathrm{E}\left[  \left\vert
\sum_{j\neq i}^{N}\xi_{k}^{j}\right\vert ^{p}\right]
\end{align*}
using (\ref{gam}). For $p\geq2,$ it follows from the Rosenthal's inequality
that,%
\[
\mathrm{E}\left[  \left\vert \sum_{j\neq i}^{N}\xi_{k}^{j}\right\vert
^{p}\right]  \leq C_{p}^{(1)}\left(  \left(  \sum_{j\neq i}^{N}\mathrm{E}%
\left\vert \xi_{k}^{j}\right\vert ^{2}\right)  ^{p/2}+\sum_{j\neq i}%
^{N}\mathrm{E}\left\vert \xi_{k}^{j}\right\vert ^{p}\right)
\]
for a constant $C_{p}^{(1)}$ only depending on $p,$ and, in fact, for $p=2$ we
have simply,%
\[
\mathrm{E}\left[  \left\vert \sum_{j\neq i}^{N}\xi_{k}^{j}\right\vert
^{p}\right]  =\sum_{j\neq i}^{N}\mathrm{E}\left\vert \xi_{k}^{j}\right\vert
^{2}.
\]
Thus, for $p\geq2,$%
\begin{align*}
\mathrm{E}\left[  \left.  \left\vert \frac{1}{N}\sum_{j=1}^{N}\xi_{k}%
^{j}\right\vert ^{p}\right\vert X_{s}^{i}\right]   &  \leq\frac{2^{2p-1}%
D_{k,\varphi}^{p}}{N^{p}}+\frac{2^{p-1}C_{p}^{(1)}}{N^{p}}\left(  \left(
\sum_{j\neq i}^{N}\mathrm{E}\left\vert \xi_{k}^{j}\right\vert ^{2}\right)
^{p/2}+\sum_{j\neq i}^{N}\mathrm{E}\left\vert \xi_{k}^{j}\right\vert
^{p}\right) \\
&  \leq\frac{2^{2p-1}D_{k,\varphi}^{p}}{N^{p}}+\frac{2^{2p-1}C_{p}%
^{(1)}D_{k,\varphi}^{p}}{N^{p/2}}+\frac{2^{2p-1}C_{p}^{(1)}D_{k,\varphi}^{p}%
}{N^{p-1}}\\
&  \leq\frac{\left(  C_{p}^{(2)}\right)  ^{p}D_{k,\varphi}^{p}}{N^{p/2}}\text{
\ \ for }N>N_{p}\text{ and some constants } C_{p}^{(2)}>0, N_{p}>0.
\end{align*}
So for any $p\geq2,$%
\begin{align*}
\left\Vert \alpha_{k}(X_{s}^{i})\frac{1}{N}\sum_{j=1}^{N}\xi_{k}%
^{j}\right\Vert _{p}^{p}  &  \leq A_{k,\alpha}^{p}\mathrm{E}\left[  \left(
1+\left\vert X_{s}^{i}\right\vert \right)  ^{p}\mathrm{E}\left[  \left.
\left\vert \frac{1}{N}\sum_{j=1}^{N}\xi_{k}^{j}\right\vert ^{p}\right\vert
X_{s}^{i}\right]  \right] \\
&  \leq A_{k,\alpha}^{p}D_{k,\varphi}^{p}\frac{\left(  C_{p}^{(2)}\right)
^{p}}{N^{p/2}}\mathrm{E}\left[  \left(  1+\left\vert X_{s}\right\vert \right)
^{p}\right]  ,
\end{align*}
hence%
\[
\left\Vert \alpha_{k}(X_{s}^{i})\frac{1}{N}\sum_{j=1}^{N}\xi_{k}%
^{j}\right\Vert _{p}\leq C_{p}^{(2)}A_{k,\alpha}D_{k,\varphi}F_{3}%
N^{-1/2}\text{ \ \ with \ \ }F_{3}:=\sup_{0\leq s\leq T}\left\Vert
1+\left\vert X_{s}\right\vert \right\Vert _{p},
\]
and further%
\[
\sum_{k=K+1}^{\infty}\left\Vert \alpha_{k}(X_{s}^{i})\gamma_{k}(s)\right\Vert
_{p}\leq F_{3}\sum_{k=K+1}^{\infty}A_{k,\alpha}\left\vert \gamma
_{k}(s)\right\vert .
\]
We thus obtain,%
\[
\left\Vert \mathfrak{a}_{K,N}(X_{s}^{i},X_{s})-\mathfrak{a}_{s}(X_{s}%
^{i})\right\Vert _{p}\leq C_{p}^{(2)}A_{\alpha}D_{\varphi}F_{3}N^{-1/2}%
+F_{3}\sum_{k=K+1}^{\infty}A_{k,\alpha}\left\vert \gamma_{k}(s)\right\vert ,
\]
that is,
\begin{align}
\mathrm{E}\left[  \left\vert \mathfrak{a}_{K,N}(X_{s}^{i},X_{s})-\mathfrak{a}%
_{s}(X_{s}^{i})\right\vert ^{p}\right]   &  \leq2^{p-1}\left(  C_{p}%
^{(2)}\right)  ^{p}A_{\alpha}^{p}D_{\varphi}^{p}F_{3}^{p}N^{-p/2}\nonumber\\
&  +2^{p-1}F_{3}^{p}\left(  \sum_{k=K+1}^{\infty}A_{k,\alpha}\left\vert
\gamma_{k}(s)\right\vert \right)  ^{p}. \label{in1}%
\end{align}
Analogously we get
\begin{align}
\mathrm{E}\left[  \left\vert \mathfrak{b}_{K,N}(X_{s}^{i},X_{s})-\mathfrak{b}%
_{s}(X_{s}^{i})\right\vert ^{p}\right]   &  \leq2^{p-1}\left(  C_{p}%
^{(2)}\right)  ^{p}A_{\beta}^{p}D_{\varphi}^{p}F_{3}^{p}N^{-p/2}\nonumber\\
&  +2^{p-1}F_{3}^{p}\left(  \sum_{k=K+1}^{\infty}A_{k,\beta}\left\vert
\gamma_{k}(s)\right\vert \right)  ^{p}. \label{in2}%
\end{align}
Now, combining the estimates (\ref{in1}) and (\ref{in2}) with (\ref{pg}),
yields for $0\leq t\leq T,$%

\begin{align*}
H(t)  &  \leq\left(  C_{p,\varphi,X}T^{p-1}+D_{p,\varphi,X}d^{p}%
T^{p/2-1}\right)  \int_{0}^{t}H(s)ds\\
&  +\left(  E_{p,\varphi,X}T^{p}+F_{p,\varphi,X}d^{p}T^{p/2}+O(N^{-1/2}%
)\right)  N^{-p/2}\\
&  +G_{p,\varphi,X}T^{p-1}\int_{0}^{T}\left(  \sum_{k=K+1}^{\infty}%
A_{k,\alpha}\left\vert \gamma_{k}(s)\right\vert \right)  ^{p}ds\\
&  +H_{p,\varphi,X}d^{p}T^{p/2-1}\int_{0}^{T}\left(  \sum_{k=K+1}^{\infty
}A_{k,\beta}\left\vert \gamma_{k}(s)\right\vert \right)  ^{p}ds
\end{align*}
with abbreviations%
\begin{align*}
C_{p,\varphi,X}  &  =2^{3p-3}D_{\varphi}^{p}B_{\alpha}^{p}+2^{3p-3}L_{\varphi
}^{p}A_{\alpha}^{p}F_{1}^{p}\\
D_{p,\varphi,X}  &  =2^{3p-3}C_{p}^{p}D_{\varphi}^{p}B_{\beta}^{p}%
+2^{3p-3}C_{p}^{p}L_{\varphi}^{p}A_{\beta}^{p}F_{1}^{p}\\
E_{p,\varphi,X}  &  =2^{3p-3}\left(  C_{p}^{(2)}\right)  ^{p}A_{\alpha}%
^{p}D_{\varphi}^{p}F_{3}^{p}\\
F_{p,\varphi,X}  &  =2^{3p-3}C_{p}^{p}\left(  C_{p}^{(2)}\right)  ^{p}%
A_{\beta}^{p}D_{\varphi}^{p}F_{3}^{p}\\
G_{p,\varphi,X}  &  =2^{3p-3}F_{3}^{p}\\
H_{p,\varphi,X}  &  =2^{3p-3}C_{p}^{p}F_{3}^{p}.
\end{align*}
Finally, the statement of the theorem follows from Gronwall's lemma by raising
the resulting inequality to the power $1/p,$ then using that $\left(
\sum_{i=1}^{q}|a_{i}|^{p}\right)  ^{1/p}\leq\sum_{i=1}^{q}|a_{i}|$ for
arbitrary $a_{i}\in\mathbb{R},$ $p,q\in\mathbb{N},$ a Minkowski type
inequality, and the observation that%
\[
\mathrm{E}\left[  \overline{\Delta_{T}^{p}}\right]  =\frac{1}{N}\sum_{i=1}%
^{N}\mathrm{E}\left[  \sup_{s\in\lbrack0,T]}\left\vert \Delta_{s}%
^{i}\right\vert ^{p}\right]  =\mathrm{E}\left[  \sup_{s\in\lbrack
0,T]}\left\vert \Delta_{s}^{\cdot}\right\vert ^{p}\right]  .
\]

\subsection{Proof of Theorem~\ref{thm:affine_expl}}

(i): Under the assumption (\ref{rder0}) the functions $H_{a^{j}}$ and $H_{b}$
are locally Lipschitz in $\mathbb{R\times R}_{\geq0}$ and, obviously, their
extensions $(p,q)\rightarrow H_{a^{j}}(p,|q|),$ $H_{b}(p,|q|)$ are locally
Lipschitz in $\mathbb{R\times R}.$ Thus, by standard ODE theory, there exists
a unique solution to the system%
\begin{align*}
G_{t}^{\prime}  &  =H_{b}^{2}\left(  A_{t},\left\vert G_{t}\right\vert
\right)  +2H_{a^{1}}\left(  A_{t},\left\vert G_{t}\right\vert \right)  G_{t}\\
A_{t}^{\prime}  &  =H_{a^{0}}\left(  A_{t},\left\vert G_{t}\right\vert
\right)  +H_{a^{1}}\left(  A_{t},\left\vert G_{t}\right\vert \right)
A_{t},\text{ \ \ }(A_{0},G_{0})=\left(  x_{0},0\right)  ,\text{ \ \ }0\leq
t<t_{\infty}\leq\infty,
\end{align*}
for some possibly finite explosion time $t_{\infty}.$ Then it can be
straightforwardly checked that this (unique)\ solution can be represented as%
\begin{align}
G_{t}  &  =\int_{0}^{t}H_{b}^{2}\left(  A_{s},\left\vert G_{s}\right\vert
\right)  e^{2\int_{s}^{t}H_{a^{1}}\left(  A_{r},\left\vert G_{r}\right\vert
\right)  dr}ds\label{loe}\\
A_{t}  &  =e^{\int_{0}^{t}H_{a^{1}}\left(  A_{s},\left\vert G_{s}\right\vert
\right)  ds}x_{0}+\int_{0}^{t}H_{a^{0}}\left(  A_{s},\left\vert G_{s}%
\right\vert \right)  e^{\int_{s}^{t}H_{a^{1}}\left(  A_{r},\left\vert
G_{r}\right\vert \right)  dr}ds,\text{\ }0\leq t<t_{\infty},\nonumber
\end{align}
whence in particular $G_{t}\geq0$ for $0\leq t<t_{\infty}.$ This proves (i).

(ii): By straightforward differentiating with respect to $t$, it follows that
(\ref{exp}) is a solution to (\ref{exp0}). Let us abbreviate in (\ref{exp})%
\[
\mathfrak{a}_{t}^{0}\equiv H_{a^{0}}\left(  A_{t},G_{t}\right)  ,\text{
\ \ }\mathfrak{a}_{t}^{1}\equiv H_{a^{1}}\left(  A_{t},G_{t}\right)  ,\text{
\ \ }\mathfrak{b}_{t}\equiv H_{b}\left(  A_{t},G_{t}\right)  ,\text{
\ \ }0\leq t<t_{\infty}.
\]
The characteristic function of $X_{t}$ in (\ref{exp}) then takes the form
\begin{equation}
\varphi_{t}(v)=\exp\left[  \mathfrak{i}v\int_{0}^{t}\mathfrak{a}_{s}%
^{0}e^{\int_{s}^{t}\mathfrak{a}_{r}^{1}dr}ds-\frac{1}{2}v^{2}\int_{0}%
^{t}\mathfrak{b}_{s}^{2}e^{2\int_{s}^{t}\mathfrak{a}_{r}^{1}dr}ds+\mathfrak{i}%
ve^{\int_{0}^{t}\mathfrak{a}_{s}^{1}ds}x_{0}\right]  . \label{cf1}%
\end{equation}
Since%
\[
\frac{e^{-\frac{(p-u)^{2}}{2q}}}{\sqrt{2\pi q}}=\frac{1}{2\pi}\int
e^{-\mathfrak{i}vu}\exp\left[  \mathfrak{i}vp-v^{2}q/2\right]  dv,
\]
we have for $j=0,1,$%
\[
H_{a^{j}}(p,q)=\frac{1}{2\pi}\int a^{j}(u)du\int\exp\left[  \mathfrak{i}%
vp-v^{2}q/2\right]  e^{-\mathfrak{i}vu}dv.
\]
It then follows that%
\begin{align}
&  H_{a^{j}}(e^{\int_{0}^{t}\mathfrak{a}_{s}^{1}ds}x_{0}+\int_{0}%
^{t}\mathfrak{a}_{s}^{0}e^{\int_{s}^{t}\mathfrak{a}_{r}^{1}dr}ds,\int_{0}%
^{t}\left(  \mathfrak{b}_{s}^{0}\right)  ^{2}e^{2\int_{s}^{t}\mathfrak{a}%
_{r}^{1}dr}ds)\nonumber\\
&  =\frac{1}{2\pi}\int a^{j}(u)du\int\varphi_{t}(v)e^{-\mathfrak{i}%
vu}dv\nonumber\\
&  =\int a^{j}(u)\mu_{t}(u)du=\mathrm{E}\left[  a^{j}(X_{t})\right]  ,\text{
\ \ }j=0,1, \label{Ha}%
\end{align}
with $\mu_{t}$ being the density of $X_{t},$ and similarly,%
\begin{equation}
H_{b}(e^{\int_{0}^{t}\mathfrak{a}_{s}^{1}ds}x_{0}+\int_{0}^{t}\mathfrak{a}%
_{s}^{0}e^{\int_{s}^{t}\mathfrak{a}_{r}^{1}dr}ds,\int_{0}^{t}\left(
\mathfrak{b}_{s}^{0}\right)  ^{2}e^{2\int_{s}^{t}\mathfrak{a}_{r}^{1}%
dr}ds)=\mathrm{E}\left[  b(X_{t})\right]  . \label{Hb}%
\end{equation}
On the other hand, in view of (\ref{loe}) and the fact that $G\geq0,$ one has%
\begin{align}
\int_{0}^{t}\left(  \mathfrak{b}_{s}^{0}\right)  ^{2}e^{2\int_{s}%
^{t}\mathfrak{a}_{r}^{1}dr}ds  &  =G_{t}\label{HH}\\
e^{\int_{0}^{t}\mathfrak{a}_{s}^{1}ds}x_{0}+\int_{0}^{t}\mathfrak{a}_{s}%
^{0}e^{\int_{s}^{t}\mathfrak{a}_{r}^{1}dr}ds  &  =A_{t},\nonumber
\end{align}
that is, by (\ref{Ha}), (\ref{Hb}), and (\ref{HH}), we obtain (\ref{MVl}) from
(\ref{exp0}).

\section{Appendix}

\label{app}

\subsection{Existence of moments}

\label{sec:moments}

\begin{thm}
\label{thm:moments} Fix some $p\geq2$ and suppose that $\mathrm{E}[|X_{0}%
|^{p}]<\infty.$ Then it holds under assumptions (AC) and (AF),
\[
\left\Vert \sup_{s\in\lbrack0,T]}\left\vert X_{s}^{\cdot,K,N}\right\vert
\right\Vert _{p}<\infty,
\]
uniformly in $K$ and $N.$
\end{thm}

\begin{proof}
Fix some $i\in\{1,\ldots,N\}$ and for every $R>0$ introduce the stopping time
\[
\tau_{i,R}=\inf\left\{  t\in\lbrack0,T]\,:\left\vert X_{t}^{i,K,N}-X_{0}%
^{i}\right\vert >R\right\}  .
\]
We obviously have
\[
\sup_{t\in\lbrack0,T]}\left\vert X_{t\wedge\tau_{i,R}}^{i,K,N}\right\vert \leq
R+\left\vert X_{0}^{i}\right\vert
\]
so that the non-decreasing function $f_{R}(t):=\left\Vert \sup_{s\in
\lbrack0,t]}\left\vert X_{s\wedge\tau_{i,R}}^{i,K,N}\right\vert \right\Vert
_{p},$ $t\in\lbrack0,T],$ is bounded by $R+\left\Vert X_{0}^{i}\right\Vert
_{p}.$ On the other hand
\begin{align*}
\sup_{s\in\lbrack0,t]}\left\vert X_{s\wedge\tau_{i,R}}^{i,K,N}\right\vert  &
\leq\left\vert X_{0}^{i}\right\vert +\int_{0}^{t\wedge\tau_{i,R}}\left\vert
\mathfrak{a}_{K,N}(X_{r}^{i,K,N},X_{r}^{K,N})\right\vert \,dr\\
&  +\sup_{s\in\lbrack0,t]}\left\vert \int_{0}^{s\wedge\tau_{i,R}}%
\mathfrak{b}_{K,N}(X_{r}^{i,K,N},X_{r}^{K,N})\,dW_{r}^{i}\right\vert \\
&  \leq\left\vert X_{0}^{i}\right\vert +\int_{0}^{t\wedge\tau_{i,R}}\left\vert
\mathfrak{a}_{K,N}(X_{r}^{i,K,N},X_{r}^{K,N})\right\vert \,dr\\
&  +\sum_{q=1}^{d}\sup_{s\in\lbrack0,t]}\left\vert \int_{0}^{s\wedge\tau
_{i,R}}\mathfrak{b}_{K,N}^{q}(X_{r}^{i,K,N},X_{r}^{K,N})\,dW_{r}%
^{i}\right\vert
\end{align*}
(cf. (\ref{defbq})). It then follows from the Minkowski and BDG inequality
that%
\begin{align*}
f_{R}(t)  &  \leq\left\Vert X_{0}\right\Vert _{p}+\int_{0}^{t}\left\Vert
1_{\{s\leq\tau_{i,R}\}}\mathfrak{a}_{K,N}(X_{s}^{i,K,N},X_{s}^{K,N}%
)\right\Vert _{p}\,ds\\
&  +dC_{p}^{BDG}\left\Vert \sqrt{\int_{0}^{t\wedge\tau_{i,R}}\left\vert
\mathfrak{b}_{K,N}(X_{s}^{i,K,N},X_{s}^{K,N})\right\vert ^{2}\,ds}\right\Vert
_{p}\\
&  \leq\left\Vert X_{0}\right\Vert _{p}+A_{\alpha}D_{\varphi}\int_{0}%
^{t}\left\Vert \left(  1+\left\vert X_{s\wedge\tau_{i,R}}^{i,K,N}\right\vert
\right)  \right\Vert _{p}\,ds\\
&  +A_{\beta}D_{\varphi}dC_{p}^{BDG}\left\Vert \sqrt{\int_{0}^{t}\left\vert
\left(  1+\left\vert X_{s\wedge\tau_{i,R}}^{i,K,N}\right\vert \right)
\right\vert ^{2}ds}\right\Vert _{p}\\
&  \leq\left\Vert X_{0}\right\Vert _{p}+A_{\alpha}D_{\varphi}\int_{0}%
^{t}\left(  1+\left\Vert \left\vert X_{s\wedge\tau_{i,R}}^{i,K,N}\right\vert
\right\Vert _{p}\right)  \,ds\\
&  +A_{\beta}D_{\varphi}dC_{p}^{BDG}\left(  \sqrt{t}+\left(  \int_{0}%
^{t}\left\Vert \left\vert X_{s\wedge\tau_{i,R}}^{i,K,N}\right\vert
^{2}\right\Vert _{p/2}\,ds\right)  ^{1/2}\right)
\end{align*}
again by the Minkowski inequality ($p\geq2$). Consequently, the function
$f_{R}$ satisfies
\[
f_{R}(t)\leq\left\Vert X_{0}\right\Vert _{p}+A_{\alpha}D_{\varphi}\int_{0}%
^{t}\left(  1+f_{R}(s)\right)  \,ds+A_{\beta}D_{\varphi}dC_{p}^{BDG}\left(
\sqrt{t}+\left(  \int_{0}^{t}f_{R}^{2}(s)\,ds\right)  ^{1/2}\right)  ,
\]
that is,%
\begin{align*}
f_{R}(t)  &  \leq\left\Vert X_{0}\right\Vert _{p}+A_{\alpha}D_{\varphi
}t+A_{\beta}D_{\varphi}dC_{p}^{BDG}\sqrt{t}\\
&  +A_{\alpha}D_{\varphi}\int_{0}^{t}f_{R}(s)\,ds+A_{\beta}D_{\varphi}%
dC_{p}^{BDG}\left(  \int_{0}^{t}f_{R}^{2}(s)\,ds\right)  ^{1/2}.
\end{align*}
By Lemma~\ref{lem: gronwall} (see Appendix) it follows that%
\begin{align}
\left\Vert \sup_{s\in\lbrack0,T]}\left\vert X_{s\wedge\tau_{i,R}}%
^{i,K,N}\right\vert \right\Vert _{p}  &  \leq2e^{\left(  2A_{\alpha}%
D_{\varphi}+A_{\beta}^{2}D_{\varphi}^{2}d^{2}\left(  C_{p}^{BDG}\right)
^{2}\right)  T}\times\label{GE}\\
&  \left(  \left\Vert X_{0}\right\Vert _{p}+A_{\alpha}D_{\varphi}T+A_{\beta
}D_{\varphi}dC_{p}^{BDG}\sqrt{T}\right)  .\nonumber
\end{align}

Now note that the stopping times $\tau_{i,R}$ are non-decreasing in $R,$ and
thus converges non-decreasingly to $\tau_{i,\infty}$ say, with $\tau
_{i,\infty}\in\lbrack0,T]\cup\{\infty\}.$ \bigskip Thus,%
\[
R\rightarrow\sup_{s\in\lbrack0,T]}\left\vert X_{s\wedge\tau_{i,R}}%
^{i,K,N}\right\vert
\]
is nondecreasing with%
\begin{equation}
\lim_{R\uparrow\infty}\sup_{s\in\lbrack0,T]}\left\vert X_{s\wedge\tau_{i,R}%
}^{i,K,N}\right\vert =\left\{
\begin{tabular}
[c]{c}%
$\sup_{s\in\lbrack0,T]}\left\vert X_{s}^{i,K,N}\right\vert $ \ \ on $\left\{
\tau_{i,\infty}=\infty\right\}  $\\
$\infty$ \ \ on $\left\{  \tau_{i,\infty}\leq T\right\}  $%
\end{tabular}
\ \ \right.  . \label{li}%
\end{equation}
Indeed, on the set $\{\tau_{i,\infty}\leq T\}$\ we have for any $R>0,$
$\left\vert X_{\tau_{i,R}}^{i,K,N}-X_{0}^{i}\right\vert \geq R$ with
$\tau_{i,R}\leq T,$ so that%
\[
\sup_{s\in\lbrack0,T]}\left\vert X_{s\wedge\tau_{i,R}}^{i,K,N}\right\vert
\geq\left\vert X_{\tau_{i,R}}^{i,K,N}\right\vert \geq\left\vert X_{\tau_{i,R}%
}^{i,K,N}\right\vert \geq R-\left\vert X_{0}^{i}\right\vert .
\]
The Fatou lemma (\ref{li}) implies (with $0:=\infty\cdot0$),
\begin{align*}
\left\Vert \lim_{R\uparrow\infty}1_{\left\{  \tau_{i,\infty}\leq T\right\}
}\sup_{s\in\lbrack0,T]}\left\vert X_{s\wedge\tau_{i,R}}^{i,K,N}\right\vert
\right\Vert _{p}  &  =\infty\cdot P\left(  \left\{  \tau_{i,\infty}\leq
T\right\}  \right) \\
&  \leq\liminf_{R}\left\Vert 1_{\left\{  \tau_{i,\infty}\leq T\right\}  }%
\sup_{s\in\lbrack0,T]}\left\vert X_{s\wedge\tau_{i,R}}^{i,K,N}\right\vert
\right\Vert _{p}\\
&  \leq\liminf_{R}\left\Vert \sup_{s\in\lbrack0,T]}\left\vert X_{s\wedge
\tau_{i,R}}^{i,K,N}\right\vert \right\Vert _{p}<\infty,
\end{align*}
because of (\ref{GE}). So $P\left(  \left\{  \tau_{i,\infty}\leq T\right\}
\right)  =0,$ i.e. $\tau_{i,\infty}=\infty$ almost surely. Again by the Fatou
lemma, (\ref{li}) then implies%
\[
\left\Vert \sup_{s\in\lbrack0,T]}\left\vert X_{s}^{i,K,N}\right\vert
\right\Vert _{p}\leq\liminf_{R}\left\Vert \sup_{s\in\lbrack0,T]}\left\vert
X_{s\wedge\tau_{i,R}}^{i,K,N}\right\vert \right\Vert _{p}<\infty,
\]
uniformly in $K$ and $N,$ because of (\ref{GE}) again.
\end{proof}

The following lemma is consequence of Gronwall's theorem.

\begin{lem}
\label{lem: gronwall} Let $f:$ $[0,T]\rightarrow\mathbb{R}_{+}$ and $\psi:$
$[0,T]\rightarrow\mathbb{R}_{+}$ be two non-negative non-decreasing functions
satisfying
\begin{equation}
f(t)\leq A\int_{0}^{t}f(s)\,ds+B\left(  \int_{0}^{t}f^{2}(s)\,ds\right)
^{1/2}+\psi(t),\quad t\in\lbrack0,T], \label{plug}%
\end{equation}
where $A,B$ are two positive real constants. Then
\[
f(t)\leq2e^{\left(  2A+B^{2}\right)  t}\,\psi(t),\quad t\in\lbrack0,T].
\]

\end{lem}

\begin{proof}
It follows from the elementary inequality $\sqrt{xy}\leq\frac{1}{2}\left(
x/B+By\right)  ,$ $x,y\geq0,$$B>0,$ that
\[
\left(  \int_{0}^{t}f^{2}(s)\,ds\right)  ^{1/2}\leq\left(  f(t)\int_{0}%
^{t}f(s)\,ds\right)  ^{1/2}\leq\frac{f(t)}{2B}+\frac{B}{2}\int_{0}%
^{t}f(s)\,ds.
\]
Plugging this into (\ref{plug}) yields
\[
f(t)\leq(2A+B^{2})\int_{0}^{t}f(s)\,ds+2\psi(t).
\]
Now the standard Gronwall inequality yields the desired result.
\end{proof}

\bibliographystyle{plain}
\bibliography{particles-1}

\begin{thebibliography}{10}

\bibitem{antonelli2002rate}
Fabio Antonelli, Arturo Kohatsu-Higa, et~al.
\newblock {Rate of convergence of a particle method to the solution of the
  McKean--Vlasov equation}.
\newblock {\em The Annals of Applied Probability}, 12(2):423--476, 2002.

\bibitem{bossy1997stochastic}
Mireille Bossy and Denis Talay.
\newblock {A stochastic particle method for the McKean-Vlasov and the Burgers
  equation}.
\newblock {\em Mathematics of Computation of the American Mathematical
  Society}, 66(217):157--192, 1997.

\bibitem{drozdov1996expansion}
AN~Drozdov and M~Morillo.
\newblock Expansion for the moments of a nonlinear stochastic model.
\newblock {\em Physical review letters}, 77(16):3280, 1996.

\bibitem{frank2004stochastic}
Till~Daniel Frank.
\newblock {\em Nonlinear {F}okker-{P}lanck equations}.
\newblock Springer Series in Synergetics. Springer-Verlag, Berlin, 2005.
\newblock Fundamentals and applications.

\bibitem{funaki1984certain}
Tadahisa Funaki.
\newblock A certain class of diffusion processes associated with nonlinear
  parabolic equations.
\newblock {\em Zeitschrift f{\"u}r Wahrscheinlichkeitstheorie und Verwandte
  Gebiete}, 67(3):331--348, 1984.

\bibitem{higham2002strong}
Desmond~J Higham, Xuerong Mao, and Andrew~M Stuart.
\newblock Strong convergence of euler-type methods for nonlinear stochastic
  differential equations.
\newblock {\em SIAM Journal on Numerical Analysis}, 40(3):1041--1063, 2002.

\bibitem{hutzenthaler2012strong}
Martin Hutzenthaler, Arnulf Jentzen, and Peter~E Kloeden.
\newblock Strong convergence of an explicit numerical method for sdes with
  nonglobally lipschitz continuous coefficients.
\newblock {\em The Annals of Applied Probability}, pages 1611--1641, 2012.

\bibitem{kloeden2017gauss}
Peter Kloeden and Tony Shardlow.
\newblock Gauss-quadrature method for one-dimensional mean-field {SDE}s.
\newblock {\em SIAM Journal on Scientific Computing}, 39(6):A2784--A2807, 2017.

\bibitem{kolokoltsov2010nonlinear}
Vassili~N Kolokoltsov.
\newblock {\em Nonlinear Markov processes and kinetic equations}, volume 182.
\newblock Cambridge University Press, 2010.

\bibitem{kostur2002nonequilibrium}
Marcin Kostur, J~{\L}uczka, and L~Schimansky-Geier.
\newblock Nonequilibrium coupled brownian phase oscillators.
\newblock {\em Physical Review E}, 65(5):051115, 2002.

\bibitem{lo2012simple}
Chi-Fai Lo and CH~Hui.
\newblock A simple analytical model for dynamics of time-varying target
  leverage ratios.
\newblock {\em The European Physical Journal B}, 85(3):1--6, 2012.

\bibitem{mckean1966class}
Henry~P McKean.
\newblock A class of markov processes associated with nonlinear parabolic
  equations.
\newblock {\em Proceedings of the National Academy of Sciences},
  56(6):1907--1911, 1966.

\bibitem{meleard1996asymptotic}
Sylvie M{\'e}l{\'e}ard.
\newblock {Asymptotic behaviour of some interacting particle systems;
  McKean-Vlasov and Boltzmann models}.
\newblock In {\em Probabilistic models for nonlinear partial differential
  equations}, pages 42--95. Springer, 1996.

\bibitem{shimizu1972phenomenological}
Hiroshi Shimizu and Takenori Yamada.
\newblock Phenomenological equations of motion of muscular contraction.
\newblock {\em Progress of theoretical physics}, 47(1):350--351, 1972.

\bibitem{szego1975orthogonal}
G~Szeg{\"o}.
\newblock Orthogonal polynomials 4th edn (providence, ri: American mathematical
  society).
\newblock 1975.

\bibitem{sznitman1991topics}
Alain-Sol Sznitman.
\newblock Topics in propagation of chaos.
\newblock In {\em Ecole d'{\'e}t{\'e} de probabilit{\'e}s de Saint-Flour XIX -
  1989}, pages 165--251. Springer, 1991.

\bibitem{zhang1997numerical}
DS~Zhang, GW~Wei, DJ~Kouri, and DK~Hoffman.
\newblock Numerical method for the nonlinear {Fokker-Planck} equation.
\newblock {\em Physical review E}, 56(1):1197, 1997.

\end{thebibliography}

\end{document}